\documentclass[10pt]{article}

\usepackage{pgf,tikz}
\usepackage{tikz}
\usepackage{pgfplots}
\pgfplotsset{compat=1.16}
\usepgfplotslibrary{fillbetween}
\usepackage{mathrsfs}
\usetikzlibrary{arrows}
\usepgfplotslibrary{fillbetween}

\usepackage{booktabs}

\usepackage{pgfplots}
\usepackage{subcaption} % For the subfigure environment
\pgfplotsset{compat=1.17}

\usepackage{diagbox}

\usepackage{graphicx}
\usepackage{caption}
\usepackage{subcaption}
\usepackage{graphicx} %Loading the package
\graphicspath{{figures/}} %Setting the graphicspat

\usetikzlibrary[patterns]
\definecolor{ududff}{rgb}{0.30196078431372547,0.30196078431372547,1.}
\definecolor{qqwuqq}{rgb}{0.,0.39215686274509803,0.}
\definecolor{xdxdff}{rgb}{0.49019607843137253,0.49019607843137253,1.}

\usepackage{fullpage, url,amsmath,amsfonts,amssymb,mathtools,mathrsfs,graphicx,  algorithm, float, sansmath,epstopdf,color,caption,enumitem,tabularx}
\usepackage[final]{pdfpages}
\usepackage{amsthm}
\usetikzlibrary{automata,topaths}
\usetikzlibrary{decorations.pathreplacing,shapes.misc}
\usepackage{fancyhdr}

\usepackage{multirow}
\usetikzlibrary{calc,arrows}

\theoremstyle{plain}

\newtheorem{theorem}{Theorem}[section]

\newtheorem{definition}{Definition}[section]

\newtheorem{lemma}{Lemma}[section]
\newtheorem{remark}{Remark}[section]

\newtheorem{example}{Example}[section]

\usepackage{blindtext}

\usepackage[colorlinks,linkcolor=blue,citecolor=red]{hyperref}

\newcommand{\rank}{\rm rank}

\usepackage{xcolor, framed}

\newcommand{\interior}{{\rm int}\kern 0.06em}

\def\<{\langle}
\def\>{\rangle}

\usepackage{ulem}

\usepackage{subcaption}

\usepackage{pict2e}

\begin{document}
\title{KKT Optimality Conditions for Multiobjective Optimal Control Problems with Endpoint and Mixed Constraints: Application to Sustainable Energy Management}
\author{Samir Adly\thanks{Laboratoire XLIM, Universit\'e de Limoges,
123 Avenue Albert Thomas,
87060 Limoges CEDEX, France.\vskip 0mm
Email: \texttt{samir.adly@unilim.fr}.} \quad and  \quad Bui Trong Kien\footnote{Department of Optimization and Control Theory,  Institute of Mathematics, Vietnam Academy of Science and Technology, 18 Hoang	Quoc Viet road, Hanoi, Vietnam; email: btkien@math.ac.vn}
}
\date{\today}
\maketitle

\begin{abstract}
In this paper, we derive first and second-order optimality conditions of KKT type for locally optimal solutions to a class of multiobjective optimal control problems with endpoint constraint and mixed pointwise constraints. We give some sufficient conditions for normality of multipliers. Namely, we show that if the linearized system is controllable or some constraint qualifications are satisfied, then the multiplier corresponding to the objective function is different from zero. To demonstrate the practical relevance of our theoretical results, we apply these conditions to a multiobjective optimal control problem for sustainable energy management in smart grids, providing insights into the trade-offs between cost, renewable energy utilization, environmental impact, and grid stability.
\end{abstract}


\section{Introduction}

We investigate the following multiobjective optimal control problem:
\begin{align}
    &\min_{u \in L^\infty([0,1], \mathbb{R}^m)} J(x,u) \quad \text{subject to:} \label{P1} \\
    & x(t) = x_0 + \int_0^t \varphi(s, x(s), u(s)) \, ds, \quad \text{for all} \ t \in [0,1], \label{P2} \\
    & h(x(1)) \leq 0, \label{P3} \\
    & g(t, x(t), u(t)) \leq 0, \quad \text{for almost all} \ t \in [0, 1], \label{P4}
\end{align}
where the objective function $J(x, u)$ is defined component-wise as
\[
    J(x, u) = \Big(J_1(x, u), J_2(x, u), \dots, J_k(x, u)\Big),
\]
with each component $J_i(x, u)$ given by
\[
    J_i(x, u) := \ell_i(x(1)) + \int_0^1 L_i(t, x(t), u(t)) \, dt, \quad i = 1, 2, \dots, k,
\]
where $L_i \colon [0, 1] \times \mathbb{R}^n \times \mathbb{R}^l \to \mathbb{R}$, $\ell_i \colon \mathbb{R}^n \to \mathbb{R}$, $\varphi \colon [0, 1] \times \mathbb{R}^n \times \mathbb{R}^m \to \mathbb{R}^n$, $h \colon \mathbb{R}^n \to \mathbb{R}^n$, $g_j \colon [0, 1] \times \mathbb{R}^n \times \mathbb{R}^m \to \mathbb{R}$,
for $i = 1, 2, \dots, k$ and $j = 1, 2, \dots, r$. 
The functions $h(x(1))$ and $g(t, x, u)$ are also defined component-wise as:
\[
    h(x(1)) = \Big(h_1(x(1)), \dots, h_n(x(1))\Big)\mbox{ and } g(t, x, u) = \Big(g_1(t, x, u), \dots, g_r(t, x, u)\Big).
\]

For vectors $y, z \in \mathbb{R}^d$, the notation $y \leq z$ implies that $z - y \in \mathbb{R}^d_+$, where
\[
    \mathbb{R}^d_+ = \big\{ (\xi_1, \xi_2, \dots, \xi_d)\in \mathbb{R}^d : \xi_i \geq 0 \ \text{for} \ i = 1, 2, \dots, d \big\}.
\]
Our primary objective is to determine a control vector function $u \in L^\infty([0, 1], \mathbb{R}^m)$ that steers the state $x \in C([0, 1], \mathbb{R}^n)$ from an initial state $x(0) = x_0$ to a terminal state $x(1)$ optimally, in the sense of Pareto optimality.

Multi-objective optimal control problems (MOCPs) have garnered significant attention due to their numerous applications in modern technology (see, for instance, \cite{Peitz}). As a result, MOCPs have been extensively studied by various researchers in recent years. For works related to the present study, we direct readers to \cite{Dmitruk1, Dmitruk2,  Kaya, Kien1, Kien3, Kien4, Kien5, Ngo, Olive, Olive2, Osmolovskiia,  Zhu} and the references therein. The main research topics in optimal control encompass the existence of optimal solutions, optimality conditions, and numerical methods for computing optimal solutions. In particular, the study of optimality conditions for multi-objective optimal control problems, as well as optimization problems in general, is a fundamental topic in optimization theory. First- and second-order optimality conditions of KKT type play a crucial role in numerical methods and the stability analysis of optimal solutions (see, for instance, \cite{Binh} and \cite{Girsanov}).

Recently, Kien et al. \cite{Kien3} derived first- and second-order necessary optimality conditions of KKT type, as well as second-order sufficient optimality conditions for multi-objective optimal control problems with free right endpoints. Furthermore, Kien et al. \cite{Kien4} established first- and second-order optimality conditions of KKT type for locally optimal solutions in the sense of Pareto for a class of multi-objective optimal problems with free right endpoint and free end-time. It is well-known that in optimal control problems without endpoint constraints, the Robinson constraint qualification is relatively easy to verify, allowing for the establishment of KKT-type optimality conditions. However, when problems include endpoint constraints, the situation becomes more complex, and verifying the Robinson constraint qualification becomes challenging. Consequently, establishing KKT-type optimality conditions becomes more difficult.

In the scalar case, several papers have addressed the normality of multipliers for problems with endpoint constraints. In \cite[Theorem 2]{Frankowska}, H. Frankowska presented a constraint qualification under which the necessary optimality conditions are normal. This constraint qualification is significant and has been studied in \cite[Theorem 9.3]{Bental} for mathematical programming problems. Additionally, V. Zeidan \cite{Zeidan} provided sufficient conditions for the normality of multipliers, demonstrating that if the linearized equation is controllable, then normality of multipliers is achieved.

The aim of this paper is to derive first- and second-order necessary optimality conditions of KKT type, as well as second-order sufficient conditions for MOCPs with endpoint constraints and mixed constraints. To address this problem, we first establish KKT optimality conditions for a specific vector optimization problem (VOP) in Banach space. The optimality conditions for VOP are established under the Robinson constraint qualification and proved directly via separation theorems and Taylor's expansion, without employing scalarization techniques or second-order tangent sets. We then derive optimality conditions for problem (P) by utilizing the results obtained for VOP. We demonstrate that if the linearized equation is controllable, then the Robinson constraint qualification is valid, ensuring that the multipliers are normal and the optimality conditions are of KKT type.

It is worth noting that although the abstract VOP serves as a good model for establishing optimality conditions for MOCPs, we cannot directly derive optimality conditions for MOCPs from those of VOP. Specifically, in MOCPs, the control variables typically belong to $L^\infty$, and consequently, the Lagrange multipliers corresponding to mixed constraints \eqref{P4} belong to $(L^\infty)^*$, which are additive measures rather than functions. Therefore, we require specific techniques and conditions to represent a functional on $(L^\infty)$ by an $L^1$-function. In \cite{Dmitruk1} and \cite{Dmitruk2}, A.V. Dmitruk and N.P. Osmolovskii imposed so-called regularity conditions on mixed constraints and employed the Yosida-Hewitt Theorem to represent a functional on $(L^\infty)^*$ by an $L^1$-function. However, in our opinion, these conditions are not easily verifiable. In our paper, we impose alternative conditions on mixed constraints that are more readily checkable. We then introduce a new technique to represent a functional on $L^\infty$ by an $L^1$-function and a novel method to derive adjoint equations.

In summary, the contributions of this paper are as follows: new results on KKT second-order necessary and sufficient conditions for Pareto solutions to MOCPs with endpoint constraints and mixed constraints, a new technique and result for establishing optimality conditions for abstract VOPs in Banach spaces with inclusion constraints, innovative techniques for representing a functional on $L^\infty$ by an $L^1$-function and deriving adjoint equations, and an application of these theoretical results to a practical problem in sustainable energy management for smart grids.

The remainder of the paper is structured as follows: In Section 2, we establish optimality conditions for VOPs in Banach spaces. Section 3 presents the main results on KKT optimality conditions for locally Pareto solutions to the MOCP. In Section 4, we apply our theoretical framework to a multiobjective optimal control problem for sustainable energy management in smart grids, demonstrating how our results can provide insights into the trade-offs between cost minimization, renewable energy utilization, environmental impact, and grid stability.

	\section{Vector optimization problem}
	\label{abstract-MP}
	Let $Z$, $E$ and $W$ be Banach spaces with the dual spaces $Z^*$, $E^*$ and $W^*$, respectively. We denote by $\tau(\mathbb{R}^m)$ the strong topology in $\mathbb{R}^m$, by $\sigma(E^*, E)$ the weak$^*$ topology in $E^*$ and by $\sigma(W^*, W)$ the weak$^*$ topology in $W^*$.  Given a Banach space $X$,  $B_X(x_0, r)$ stands for the open ball with center $x_0$ and radius $r$. Given a subset  $M$ in $X$. We shall denote by $\mathrm{int}M$ and $\overline{M}$ the interior and the closure of $M$, respectively. 
	
	Let $f=(f_1, f_2,..., f_k)\colon Z\to\mathbb{R}^k$,   $F\colon Z\to E$,  $G: Z\to W$ and $H=(H_1, H_2,..., H_n):Z\to \mathbb{R}^n$  be mappings  and $K$ be a nonempty convex set in $W$. 	 We consider the following vector optimization problem. 
%\begin{align*}
%		&f(z)\to \min\\
%		&\text{s.t.}\\
%(VOP)\quad \quad &F(z)=0,\\  
%		&H(z)\leq 0,\\
%		&G(z)\in K.
%\end{align*}
$${\rm (VOP)}\left\{
\begin{array}{l}
f(z)\to \min \mbox{ s.t. }\\
F(z)=0,\; H(z)\leq 0\mbox{ and }G(z)\in K.
\end{array}
\right.
$$

We shall denote by $\Sigma$ the feasible set of  ${\rm (VOP)}$, that is,
	$$
	\Sigma=\{z\in Z\;|\; F(z)=0, H_i(z)\leq 0 \mbox{ and } G(z)\in K, i=1,2,..., n\}.
	$$ 
	To derive optimality conditions for ${\rm (VOP)}$ we shall need some tools and concepts  of variational analysis.	Let  $S$  be a  nonempty  set of $X$ and $\bar x\in	\bar S$. Then the set
	\begin{align*}
		T(S, \bar x)&:=\left\{h\in X\;|\; \exists t_n\to 0^+, \exists h_n\to h, \bar x+t_nh_n\in S, \; \forall n\in\mathbb{N}\right\},
	\end{align*}
	is called  {\it the contingent cone} to $S$ at $\bar x$. It is well-known that when $S$ is convex, then
	$$
 T(S, \bar x)=\overline{S(\bar x)},
        $$
	where 
	$$
	S(\bar x):=\mathrm{cone}\,(S-\bar x)=\{\lambda (h-\bar x)\;|\; h\in S, \lambda>0\}.
	$$ When $S$ is a convex set, the
	normal cone to $S$ at $\bar x$ is defined by
	\begin{equation*}
		N(S, \bar x):=\{x^*\in X^*\;|\; \langle x^*, x-\bar x\rangle\leq 0, \ \ \forall x\in S\},
	\end{equation*}
	or, equivalently,
	\begin{equation*}
		N(S, \bar x)=\{x^*\in X^*\;|\; \langle x^*, h\rangle\leq 0,\ \ \forall h\in T(S, \bar x)\}.
	\end{equation*}
Given a feasible point $z_0 \in \Sigma$,  we denote by $f'(z_0)$ or $Df(z_0)$ the first order derivative of $f$ at $z_0$ and by $f''(z_0)$ or $D^2 f(z_0)$ the second-order derivative of $f$ at $z_0$. Let us impose the following assumptions.

\begin{enumerate}
	\item[$(A1)$] ${\rm int} (K)\neq \emptyset$;
	
	\item[$(A2)$] The mappings $f, F$ and $G$ are of class $C^2$ around $z_0$;
	
	\item[$(A3)$] The range of $DF(z_0)$ is closed in $E$. 
	
\end{enumerate}
%\medskip 
By defining $\tilde{K} = {0} \times (-\infty, 0]^n \times K$ and $\tilde{G} = (F, H, G)$, the constraint of $\text{(VOP)}$ is equivalent to the constraint $\tilde{G}(z) \in \tilde{K}$.
Recall that a feasible point  $z_0$ is said to satisfy the Robinson constraint qualification  if
\begin{align}\label{RobinsonCond1}
	0\in {\rm int} \{\tilde{G}(z_0) + D\tilde{G}(z_0)Z -\tilde{K}\}. 
\end{align} By Proposition 2.95 in \cite{Bonnans2000}, if $K$ is closed, then condition \eqref{RobinsonCond1} is equivalent to
\begin{align}\label{RobinsonCond2}
	D\tilde{G}(z_0)Z -{\rm cone}(\tilde{K}-\tilde{G}(z_0))=E\times\mathbb{R}^n\times W.  
\end{align} This means that for any $(e, \xi, w)\in E\times \mathbb{R}^n\times W$, there exist $z\in Z$, $v\in {\rm cone}((-\infty, 0]^n-H(z_0))$ and $v'\in {\rm cone}(K-G(z_0))$ such that 
\begin{equation} \label{RobinsonCond3}
%	\begin{cases}
%		DF(z_0)z=e,\\
%		DH(z_0)z-v=\xi,\\
%		DG(z_0)z-v'=w.
%	\end{cases}
		DF(z_0)z=e,\;
		DH(z_0)z-v=\xi \mbox{ and }
		DG(z_0)z-v'=w.
\end{equation}  
The Robinson constraint qualification guarantees that the necessary optimality conditions are of KKT type. Therefore, we require that ${\rm (VOP)}$ satisfies the following assumption.

\medskip

\noindent $(A4)$ $K$ is closed and  $z_0$ satisfies the Robinson constraint qualification \eqref{RobinsonCond2}.

Note that if $(A4)$ is satisfied, then $DF(z_0)$ is surjective. Hence $DF(z_0)Z=E$ and so $(A3)$ is fulfilled.  If we define
$
\hat K= (-\infty, 0]^n \times K,
$ then ${\rm int}(\hat K)\neq \emptyset$.  According to \cite[Corollary 2.101,  page 70]{Bonnans2000},  the following condition is equivalent to condition  \eqref{RobinsonCond1}.
\begin{align*}	
(A5)\quad \quad \quad
	\begin{cases}
		(i) \quad  DF(z_0) \  \ {\rm is \ surjective} \\
		(ii) \quad \exists \widetilde z \in {\rm Ker}(DF(z_0))\  {\rm such \ that}\  H_i(z_0)+  DH_i(z_0)\widetilde z <0\ \forall i=1,2,..., n\\
		\text{and}\  G(z_0) +  DG(z_0) \widetilde z  \in {\rm int}(K).
	\end{cases}
\end{align*}It is worth noting that $(A5)$ is a variant of the Mangasarian-Fromovitz condition.

\begin{definition}
The vector $z_0\in \Sigma$ is a locally weak Pareto solution of problem ${\rm (VOP)}$ if there exists $\epsilon>0$ such that $
f(z)-f(z_0)\notin - \mathrm{int}(\mathbb{R}^k_+) \quad \forall z\in B_Z(z_0, \epsilon)\cap \Sigma.$\\
If $f(z)-f(z_0)\notin - \mathbb{R}^k_+\setminus\{0\}$ for all $z\in B_Z(z_0, \epsilon)\cap \Sigma$, we say $z_0$ is a locally Pareto solution to ${\rm (VOP)}$.
\end{definition}

Obviously, a locally Pareto solution is also a locally weak Pareto solution. 

\noindent To deal with second-order optimality conditions, we need the so-called critical cone, which is defined as follows. Let $\mathcal{C}_0[z_0]$ be the set of vectors $d \in Z$ satisfying the conditions:
\begin{enumerate}[label=($a_{\arabic*}$)]
    \item $Df(z_0)d \in -\mathbb{R}^k_+$
    \item $DF(z_0)d = 0$
    \item $DH_i(z_0)d \leq 0$ for $i \in I_0$, where $I_0 := \{i \in \{1,2,\ldots,n\} : H_i(z_0) = 0\}$
    \item $DG(z_0)d \in \operatorname{cone}(K - G(z_0))$
\end{enumerate}
\noindent Then the closure of $\mathcal{C}_0[z_0]$ in $Z$ is called the critical cone at $z_0$ and denoted by $\mathcal{C}[z_0]$. Each vector $d \in \mathcal{C}[z_0]$ is called a critical direction.

\begin{lemma}\label{Lemma-key1} Suppose that $\bar z\in \Sigma$  is a locally weak  Pareto solution of ${\rm (VOP)}$ under which  $(A_1)$, $(A_2)$ and $(A3)$ are valid,  and $D F(z_0)$ is surjective.  Then for each  $d\in \mathcal{C}_0[z_0]$, the system
	\begin{align}
		&f'_i(z_0)z+\frac{1}2f''_i(z_0)d^2< 0, \ \ \forall i\in\{1,2,..., k\},  \label{PrimalOptimCond1}\\
		&H'_i(z_0)z+\frac{1}22H''_i(z_0)d^2< 0, \ \ \forall i\in I_0,  \label{PrimalOptimCond1+}\\
		&DF(z_0)z +\frac{1}2D^2 F(z_0)(d,d) =0, \label{PrimalOptimCond2}\\
		&D G(z_0)z+\frac{1}2D^2G(z_0))(d,d)\in {\rm cone}({\rm int} (K)-G(z_0)) \label{PrimalOptimCond3}
	\end{align}
	has no solution $z\in Z$.
\end{lemma}
\begin{proof} On the contrary, we suppose that the system \eqref{PrimalOptimCond1}-\eqref{PrimalOptimCond3} has a solution $z\in Z$. For each $\epsilon>0$, we define $z_\epsilon= z_0 +\epsilon d +\epsilon^2 z.$ By Taylor expansion, we have 
	\begin{align*}
		F(z_\epsilon)= F(z_0) +\epsilon DF(z_0)d +	\epsilon^2(D F(z_0)z+\frac{1}2 D^2 F(z_0)(d,d))+ o(\epsilon^2)=o(\epsilon^2).
	\end{align*} Since $DF(z_0)Z=E$, the Ljusternik theorem (see \cite{Ioffe}, page 30) implies that, there exist a neighborhood $B(z_0, \gamma)$ of $z_0$ and a mapping $h: B(z_0, \gamma)\to Z$ such that for all $\hat z\in B(z_0, \gamma)$ we have 
	$$
	F(\hat z + h(\hat z))=0,\quad \|h(\hat z)\|\leq C\|F(\hat z)\|.
	$$ Taking $\hat z=z_\epsilon$ and $h_\epsilon=h(z_\epsilon)$, we see that, there exists $\epsilon_0>0$ such that 
	\begin{align}\label{Feas-Cond1}
		F(z_\epsilon + h_\epsilon)=0,\ \|h_\epsilon\|=o(\epsilon^2)\quad \text{for all}\quad 0<\epsilon<\epsilon_0.
	\end{align} Putting
	$$
	z_\epsilon'= z_\epsilon + h_\epsilon=z_0+\epsilon d +\epsilon^2 (z+\frac{1}{\epsilon^2}h_\epsilon)
	$$ and using a Taylor's expansion, we get	
	\begin{align*}
		f_i(z'_\epsilon)= f_i(z_0) +\epsilon f_i'(z_0)d +	\epsilon^2(f_i'(z_0)z+\frac{1}2f_i''(z_0)(d,d) +\frac{o(\epsilon^2)}{\epsilon^2})\quad\text{for}\quad 0<\epsilon<\epsilon_0. 
	\end{align*} Since 
	$$
	f_i'(z_0)z+\frac{1}2f_i''(z_0)(d,d)<0,
	$$ there exists $\epsilon_i\in (0, \epsilon_0)$ such that for all $\epsilon\in (0, \epsilon_i)$, we have 
	$$
	f_i'(z_0)z+\frac{1}2f_i''(z_0)(d,d) +\frac{o(\epsilon^2)}{\epsilon^2}<0.
	$$ From this and  the fact that $f_i'(z_0)d\leq 0$, we have 
	\begin{align*}
		f_i(z'_\epsilon)-f_i(z_0)<0 \quad \forall \epsilon <\epsilon_i,\  i=1,2,..., m.
	\end{align*} Let $\epsilon'=\min\{\epsilon_1,..., \epsilon_m\}$.  It follows that
	\begin{align}\label{Pareto-Cond2}
		f(z'_\epsilon)-f(z_0)\in -{\rm int}\mathbb{R}_{+}^m \quad \text{for}\quad 0<\epsilon<\epsilon'. 
	\end{align} Also by  a Taylor's expansion, for each $i\in\{1,2,..., n\}$ we have
\begin{align*}
H_i(z'_\epsilon)= H_i(z_0) +\epsilon H'_i(z_0)d +	\epsilon^2(H'_i(z_0)z+\frac{1}2H''_i(z_0)(d,d) +\frac{o(\epsilon^2)}{\epsilon^2})\quad\text{for}\quad 0<\epsilon<\epsilon'. 
\end{align*} If $i\in I_0$ then  
$$
 H'_i(z_0)z+\frac{1}2 H''_i(z_0)(d,d)<0,
$$ and $H_i(z_0) =0$. Therefore, there exists $\epsilon'_i\in (0, \epsilon')$ such that for all $\epsilon\in (0, \epsilon'_i)$ one has 
$$
H'_i(z_0)z+\frac{1}2H''_i(z_0)(d,d) +\frac{o(\epsilon^2)}{\epsilon^2}<0.
$$ Hence 
$$
H_i(z'_\epsilon)\leq  +\epsilon H'_i(z_0)d \leq 0\quad \forall \epsilon\in (0, \epsilon'_i). 
$$ If $i\notin I_0$ then $H_i(z_0)<0$. By the continuity of $H_i$,   there exists $\epsilon''_i<\epsilon_i'$ such that $H_i(z'_\epsilon)<0$ for all $\epsilon\in (0, \epsilon_i'')$.  Hence $H_i(z'_\epsilon)<0$ for all $i\in\{1,2,..., n\}$ and $\epsilon\in(0, \epsilon'')$ with $\epsilon''=\min(\epsilon_1'',..., \epsilon_n'')$.  Besides,  \eqref{Pareto-Cond2} is valid for all $\epsilon\in (0, \epsilon'')$.  Using a Taylor's expansion again, we have 
	\begin{align*}
		G(z'_\epsilon)= G(z_0) +\epsilon D G(z_0)d +	\epsilon^2\big(D G(z_0)z+\frac{1}2D^2 G(z_0)(d,d) +\frac{o(\epsilon^2)}{\epsilon^2}\big).	
	\end{align*} Since $D G(z_0)d\in {\rm cone}(K-G(z_0))$, there exist $\lambda_1>0$ and $v_1\in K$ such that $\nabla G(z_0)d=\lambda_1(v_1-G(z_0))$ for some $v_1\in K$. Also, since 
	$$
	D G(z_0)z+\frac{1}2 D^2 G(z_0)(d,d)\in {\rm cone}({\rm int}(K)-G(z_0)),
	$$ there exist $\lambda_2>0$ and $v_2\in{\rm int}(K)$ such that 
	$$
	DG(z_0)z+\frac{1}2D^2 G(z_0)(d,d)=\lambda_2(v_2-G(z_0)). 
	$$ Hence,
	\begin{align*}
		G(z'_\epsilon)&=G(z_0)+\epsilon \lambda_1(v_1-G(z_0)) + \epsilon^2 \lambda_2(v_2-G(z_0))+ o(\epsilon^2)\\
		&=(1-\epsilon\lambda_1-\epsilon^2\lambda_2)G(z_0)+\epsilon\lambda_1 v_1 +\epsilon^2\lambda_2\big(v_2 +\frac{o(\epsilon^2)}{\lambda_2\epsilon^2}\big).
	\end{align*} Since $ v_2\in {\rm int}(K)$, there exists $\epsilon_*\in (0, \epsilon'')$  such that $v_2 +\frac{o(\epsilon^2)}{\lambda_2\epsilon^2}\in {\rm int}(K)$ and $1-\epsilon\lambda_1-\epsilon^2\lambda_2>0$ for all $\epsilon\in (0, \epsilon_*)$. This implies that 
	\begin{align*}
		G(z'_\epsilon)&\in (1-\epsilon\lambda_1-\epsilon^2\lambda_2)K + \epsilon\lambda_1 K +\epsilon^2\lambda_2 K\\
		&\subset (1-\epsilon\lambda_1-\epsilon^2\lambda_2)K + (\epsilon\lambda_1 +\epsilon^2\lambda_2)( \frac{\epsilon\lambda_1}{\epsilon\lambda_1 +\epsilon^2\lambda_2} K + \frac{\epsilon^2\lambda_2}{\epsilon\lambda_1 +\epsilon^2\lambda_2}K)\\
		&\subset (1-\epsilon\lambda_1-\epsilon^2\lambda_2)K + (\epsilon\lambda_1 +\epsilon^2\lambda_2) K\subset K. 
	\end{align*} Combining this with \eqref{Feas-Cond1} and the fact $H(z'_\epsilon)\leq 0$, we have $z'_\epsilon\in\Sigma$ for all $\epsilon\in (0, \epsilon_*)$. By this and   \eqref{Pareto-Cond2}, we conclude that $z_0$ is not a locally weak  Pareto solution of ${\rm (VOP)}$ which is absurd. The lemma is proved.
\end{proof}
\medskip
 Let us assume that 
\begin{align}
	{\mathcal L}\left( z, \lambda, e^*, l, w^* \right) = \lambda^T f(z) + \langle e^*, F(z)\rangle + l^T H(z)+ \langle w^*, G(z)\rangle 
\end{align} is  the Lagrange function  associated with the problem ${\rm (VOP)}$, where $\lambda\in\mathbb{R}^k$, $e^* \in E^*$, $l\in\mathbb{R}^n$ and $w^* \in W^*$. We say that vector $(\lambda, e^*, l, w^*)\in \mathbb{R}^k\times E^*\times\mathbb{R}^n\times W^*$ are Lagrange multipliers at $z_0$ if the following conditions are fulfilled:
\begin{align}
	&D_z  {\mathcal L}\left( z_0,\lambda,  e^*, l,  w^* \right)= \lambda^T Df(z_0) + DF(z_0)^*e^* + DH(z_0)^* l^T+ DG(z_0)^* w^*=0,\label{L1}\\
	&\lambda\geq 0,\label{L2}\\
	& l_i\geq 0,\ l_i H_i(z_0)=0, i=1,2,.., n,\label{L3}\\
	& w^*\in N(K, G(z_0)),\label{L4}
\end{align} where $l=(l_1, l_2,..., l_n)$. We denote by $\Lambda[z_0]$ the set of Lagrange multipliers at $z_0$. In addition,  if $|\lambda|=1$, then we say $(\lambda, e^*, l, w^*)$ are normal. We shall denote by $\Lambda_*[z_0]$ the set of normal Lagrange multipliers at $z_0$.\\
We now have the following result on first and second-order necessary optimality conditions for ${\rm (VOP)}$.
\begin{theorem}\label{Theorem-SONC-VOP} Suppose that $z_0$ is a local weak Pareto solution to ${\rm (VOP)}$. Then   the following assertions are fulfilled:
\noindent $(a)$ If $(A1)- (A3)$ are satisfied, then  $\Lambda[z_0]$ is nonempty and  
\begin{align}\label{VOP-SOC1}
	\sup_{(\lambda, e^*,l,  w^*)\in\Lambda[z_0]} D^2_z\mathcal{L}(z_0, \lambda, e^*, l,  w^*)[d,d]\geq 0,\quad \forall d\in \mathcal{C}_0[z_0].   
\end{align}
\noindent $(b)$ If $(A1), (A2)$ and $(A4)$ or $(A1), (A2)$ and $(A5)$ are satisfied, then $\Lambda_*[z_0]$  is nonempty and compact  in the topology $\sigma(\mathbb{R}^k, \mathbb{R}^k)\times\sigma(E^*, E)\times\sigma(W^*, W)$,  and  
\begin{align}\label{VOP-SOC2}
\max_{(\lambda, e^*,l,  w^*)\in\Lambda_*[z_0]} D^2_z\mathcal{L}(z_0, \lambda, e^*, l,  w^*)[d,d]\geq 0,\quad \forall d\in \mathcal{C}[z_0].   
\end{align}
\end{theorem}
\begin{proof} $(a)$. We take any $d\in \mathcal{C}_0[z_0]$ and consider two cases.
	 
\noindent {\it Case 1}. $D F(z_0)Z\neq E$. Then there is a point $e_0\in E$ and $e_0\notin DF(z_0)Z$. Since $DF(z_0)Z$ is a closed subspace,  the separation theorem (see \cite[Theorem 3.4]{Rudin}) implies that there exists a nonzero functional $e^*\in E^*$ which separate $e_0$ and $\nabla F(z_0)Z$. By a simple argument, we see that $DF(z_0)^* e^*=0$. Putting $\lambda=0, l=0, w^*=0$, we get $(0, e^*, 0, 0)\in\Lambda(z_0)$. If $e^*DF(z_0)d^2\leq 0$, then we replace $e^*$ by $-e^*$. Then $(0, -e^*, 0, 0)$ satisfies the conclusion of $(a)$. 
	
\noindent {\it Case 2}. $D F(z_0)Z= E$.

Let $n_0=|I_0|$. We define a set $S$ which consists of vectors $(\mu, \gamma,  e, w)\in\mathbb{R}^k\times\mathbb{R}^{n_0}\times E\times W$ such that there exists $z\in Z$ satisfying
\begin{align}
	&f'_i(z_0)z+\frac{1}2f''_i(z_0)d^2<\mu_i, i=1,2,.., k\\
	&H'_j(z_0)z+\frac{1}2H''_j(z_0)d^2<\gamma_j, j\in I_0,\\
	&D F(z_0)z+\frac{1}2 D^2F(z_0)d^2=e \label{C1}\\
	&D G(z_0)z+ \frac{1}2 D^2G(z_0)d^2 -w\in{\rm cone}({\rm int}(K)-G(z_0)). 
\end{align} It is clear that $S$ is convex. We now show that $S$ is open. In fact, take $(\hat\mu, \hat\gamma, \hat e,\hat w)\in S$ corresponding to $\hat z$. Choose $\epsilon>0$ such that 

$$
f'_i(z_0)\hat z+\frac{1}2f''_i(z_0)d^2<\hat\mu_i-\epsilon,\quad\forall i=1,2,.., k,\quad
H'_j(z_0)\hat z+\frac{1}2H''_j(z_0)d^2<\hat\gamma_j-\epsilon,\quad \forall j\in I_0. 
$$
Then there is $\delta>0$ such that  for all $z\in B(\hat z, \delta)$ one has

$$
f'_i(z_0)z+\frac{1}2f''_i(z_0)d^2<\hat\mu_i-\epsilon\quad\forall i=1,2,.., k,\quad H_j'(z_0)z+\frac{1}2H''_j(z_0)d^2<\hat\gamma_j-\epsilon\quad \forall j\in I_0. 
$$
%%%

 It follows that  
\begin{align*}
&f'_i(z_0)z+\frac{1}2f''_i(z_0)d^2<\mu_i,\quad\forall z\in B(\hat z, \delta)\ {\rm and}\ |\mu_i-\hat\mu_i|<\epsilon,  i=1,2,.., k,\\
&H'_j(z_0)z+\frac{1}2H''_j(z_0)d^2<\gamma_j,\quad \forall z\in B(\hat z, \delta)\ {\rm and}\ |\gamma_i-\hat\gamma_i|<\epsilon, j\in I_0. 
\end{align*}
Since ${\rm cone}({\rm int} K-G(z_0))$ is an open convex cone and
$$
DG(z_0)\hat z+\frac{1}2 D^2G(z_0)d^2-\hat w \in {\rm cone}({\rm int}(K)-G(z_0)),
$$ the continuity of $DG(z_0)$ implies that there exit  balls  $B(\hat z, \delta)$ and   $B(\hat w, r)$  such that 
$$
DG(z_0)z +\frac{1}2 D^2G(z_0)d^2 -w \in {\rm cone}({\rm int}(K)-G(z_0))\quad \forall (z,w)\in B(\hat z, \delta)\times B(\hat w, r).
$$ Since $DF(z_0)$ is surjective, $DF(z_0)[B(\hat z, \delta)]+\frac{1}2 D^2F(z_0)d^2$ is open. Hence, there exists a number $\alpha>0$ such that 
$$
B(\hat e, \alpha)\subset DF(z_0)[B(\hat z, \delta)]+ \frac{1}2 D^2G(z_0)d^2. 
$$ Thus, for any $(\mu,\gamma,  e, w)\in  B_{\mathbb{R}^k}(\hat \mu, \epsilon)\times B_{\mathbb{R}^{n_0}}(\hat\gamma, \epsilon)\times B_E(\hat e, \alpha)\times B_W(\hat w, r)$, there exists $z\in B(\hat  z, \delta)$ such that 

\begin{align*}
	&f'_i(z_0)z+\frac{1}2f''_i(z_0)d^2<\mu_i,\;\forall   i=1,2,.., k,\quad H'_j(z_0)z+\frac{1}2H''_j(z_0)d^2<\gamma_j\quad \forall   j\in I_0.\\
	&e=DF(z_0)z +\frac{1}2D^2F(z_0)d^2,\quad DG(\bar z)z +\frac{1}2D^2G(z_0)d^2 -w\in {\rm cone}({\rm int}(K)-G(\bar z)).
\end{align*}
This means
$$
 B_{\mathbb{R}^k}(\hat \mu, \epsilon)\times B_{\mathbb{R}^{n_0}}(\hat\gamma, \epsilon)\times B_E(\hat e, \alpha)\times B_W(\hat w, r) \subset S.
$$ Hence $S$ is open.  By Lemma \ref{Lemma-key1}, we have  $(0, 0, 0, 0)\notin S$. By  the separation theorem (see \cite[Theorem 1, p. 163]{Ioffe},  there exists a nonzero vector 
$(\lambda, l, e^*,  w^*)\in \mathbb{R}^k\times\mathbb{R}^{n_0}\times E^*\times \times W^*$ such that 
\begin{align}\label{Seperation2}
	\lambda\mu^T + l\gamma \langle e^*, e\rangle  +\langle w^*, w\rangle\geq 0\quad \forall (\mu, \gamma, e,  w)\in S. 
\end{align} 
Fix any $z\in  Z$,  $w'\in{\rm cone }({\rm int }K-G(z_0))$,  $r_i>0$ and $r'_j>0$ . Set

\begin{align*}
	& \mu_i=r_i+ f'_i(z_0)z +\frac{1}2 f''_i(z_0)d^2,\quad \gamma_j=r_j'+  H'_j(z_0)z +\frac{1}2H''_j(z_0)d^2,\\
	&e=D F(z_0)z +\frac{1}2D^2 F(z_0)d^2,\quad  w=D G(z_0)z+ \frac{1}2D^2 G(z_0)d^2-w'
\end{align*} 
Then $(\mu,\gamma,  e, w)\in S$. From this and  \eqref{Seperation2}, we get
\begin{align*}
&\sum_{i=1}^k\lambda_i(r_i+ f'_i(z_0)z +\frac{1}2 f''_i(z_0)d^2)+ \sum_{j\in I_0} l_j(r'_j+ H'_j(z_0)z +\frac{1}2 H''_j(z_0)d^2)\\
&+ \langle e^*, DF(z_0)z +\frac{1}2D^2 F(z_0)d^2\rangle +\langle w^*, DG(z_0)z +\frac{1}2 D^2G(z_0)d^2\rangle -\langle w^*, w'\rangle\geq 0.
\end{align*} If there exists $i_0$ such that $\lambda_{i_0}<0$ then by letting $r_{i_0}\to+\infty$, the term on the left hand side approach to $-\infty$ which is impossible. Hence we must have $\lambda_i\geq 0$ for all $i=1,2,.., k$. Similarly, we have $l_j\geq 0$ for all $j\in I_0$. 

By letting $r_i\to 0$ and $r_j'\to 0$,  we get 
\begin{align}\label{Seperation3}
	\sum_{i=1}^k\lambda_i f'_i(z_0)+\sum_{j\in I_0}l_jH'_j(z_0) + DF(z_0)^* e^* + DG(z_0)^* w^*=0
\end{align} and 
$$
\frac{1}2\big(\sum_{i=1}^k\lambda_i f''_i(z_0)d^2+ \sum_{j\in I_0} l_i H''_i(z_0) d^2+  \langle e^*, D^2F(z_0)d^2\rangle +
 \langle w^*, D^2G(z_0)d^2\rangle \big)\geq \langle w^*, w'\rangle 
$$ for all $w'\in{\rm cone}({\rm int}(K)-G(z_0))$. It follows that 
$$
\langle w^*, w'\rangle\leq 0,  \forall w'\in{\rm cone}({\rm int}(K)-G(z_0)).
$$ This implies that $w^*\in N(K, G(z_0))$ because $ K\subseteq\bar K=\overline{{\rm int}(K)}$.  By letting $w'\to 0$, we get 
\begin{align}\label{Saperation4}
\sum_{i=1}^k\lambda_if''_i(z_0)d^2+ \sum_{j\in I_0} l_j H''_i(z_0) d^2+  \langle e^*,D^2F(z_0)d^2\rangle +
 \langle w^*, D^2G(z_0)d^2\rangle \geq 0.
\end{align} We now take $l_j=0$ for $j\in\{1,2,..., n\}\setminus I_0$. Then we have $l_j\geq 0$ and $l_j H_j(z_0)=0$ for all $j=1,2,.., n$ and \eqref{Seperation3} and \eqref{Saperation4} become
\begin{align} \label{Saperation5}
	\sum_{i=1}^k\lambda_i f'_i(z_0)+\sum_{j=1}^n l_jH'_j(z_0) + DF(z_0)^*e^* + DG(z_0)^*w^*=0
\end{align} and 
\begin{align}\label{Saperation6}
	\sum_{i=1}^k\lambda_if''_i(z_0)d^2+ \sum_{j=1}^n l_jH''_i(z_0) d^2+  \langle e^*. D^2F(z_0)d^2\rangle +
	\langle w^*, D^2G(z_0)d^2\rangle \geq 0.
\end{align} Assertion $(a)$ is proved. 

\noindent $(b)$. We claim that $\lambda\neq 0$. For this we consider two cases. 

\noindent {\it Case 1.} $(A1), (A2)$ and $(A4)$ are satisfied.  

If  $\lambda=0$, then we have 
$$
DH(z_0)^* l^T+ DF(z_0)^* e^* + DG(z_0)^* w^*= 0.
$$ Take any $\xi\in\mathbb{R}^n$, $e\in E$ and $w\in W$. By $(A4)$, there exist $z\in Z$, $v\in {\rm cone}((-\infty, 0]^n-H(z_0))$ and $v'\in {\rm cone}(K-G(z_0))$ satisfying \eqref{RobinsonCond3}. From the above, we have 
\begin{align*}
0&= l^TDH(z_0)z + \langle e^*, DF(z_0)z\rangle + \langle w^*, DG(z_0)z\rangle = l^T \xi + l^T v + \langle e^*, e\rangle +\langle w^*, w\rangle +\langle w^*, v'\rangle\\
& \leq l^T \xi +  \langle e^*, e\rangle +\langle w^*, w\rangle. 
\end{align*} Here we used the fact that $w^*\in N(K, G(z_0))$ and $l\in N((-\infty, 0]^n, H(z_0))$. As $\xi, e, w$ are arbitrary, we obtain $l=0, e^*=0$ and $w^*=0$ which is impossible. Hence we must have $\lambda\neq 0$ and so $\Lambda_*[z_0]\neq \emptyset$.  

\noindent {\it Case 2}. $(A1), (A2)$ and $(A5)$ are fulfilled. 

If $\lambda=0$, then we have 
$$
DH(z_0)^* l^T+ DF(z_0)^* e^* + DG(z_0)^* w^*= 0.
$$ Let $\tilde z$ be a vector  satisfying  $(A5)$. Then we have 
$$
l^TDH(z_0)\tilde z+  \langle e^*,  DF(z_0)\tilde z\rangle   + \langle w^*,  DG(z_0)^* \tilde z\rangle =0.
$$  This implies that 
$$
l^T DH(z_0)\tilde z + \langle w^*, DG(z_0) \tilde z\rangle=0,  
$$ where
$$
DH(z_0)\tilde z\in (-\infty, 0)^n -H(z_0)={\rm int}( (-\infty, 0]^n-H(z_0))
$$ and 
$$
DG(z_0) \tilde z \in {\rm int}(K)- G(z_0)={\rm int}(K-G(z_0)). 
$$ Fixing any $\xi\in\mathbb{R}^n$ and $w\in W$, we see that, there exist $\tau>0$ and  $s>0$ small enough such that 
$$
\tau\xi + DH(z_0)\tilde z \in (-\infty, 0]^n-H(z_0)
\mbox{ and }
sw+ DG(z_0) \tilde z \in K-G(z_0).
$$  Since $l\geq 0$, $l^T H(z_0)=0$ and $w^*\in N(K, G(z_0))$, we have 
$$
\tau l^T\xi + s\langle w^*, w\rangle=l^T(\tau \xi+ DH(z_0)\tilde z)+\langle w^*, sw+ DG(z_0) \tilde z \rangle\leq 0.  
$$ Hence $l^T\xi\leq 0$ for all $\xi\in\mathbb{R}^n$ and  $\langle w^*, w\rangle \leq 0$ for all $w\in W$. This implies that $l=0$ and  $w^*=0$. Consequently, $DF(z_0)^* e^*=0$. Since $DF(z_0)$ is surjective, we get $e^*=0$. Hence $(\lambda, e^*, l,  w^*)=(0,0,0, 0)$ which is absurd. The claim is justified. 

The boundeness of $\Lambda_*[z_0]$ is proved analogously with Lemma 3.3 in \cite{Kien3}.  It remains to prove \eqref{VOP-SOC2}. Let us consider function
$$
\psi(\lambda,  e^*, l,  w^*, d):=D_z^2\mathcal{L}(z_0, \lambda, e^*, l, w^*)(d,d).
$$ It is continuous in the topology $\tau(\mathbb{R}^k)\times\sigma(E^*, E)\times\tau(\mathbb{R}^n)\times\sigma(W^*, W)\times \tau(Z)$, where $\tau(\mathbb{R}^k)$ and $\tau(Z)$ are strong topologies in $\mathbb{R}^k$ and $Z$, respectively. As $\Lambda_*[z_0]$ is compact in the topology $\tau(\mathbb{R}^k)\times\sigma(E^*, E)\times \tau(\mathbb{R}^n)\times\sigma(W^*, W)$,  the function 
$$
d\mapsto \max_{(\lambda, e^*,l, w^*)\in\Lambda (z_0)}D_z^2\mathcal{L}(z_0, \lambda, e^*, l, w^*)(d, d)
$$ is continuous (see \cite[Theorem 1 and 2, p. 115]{Berge}). Let  $d\in \mathcal{C}[z_0]$. Then there exists a sequence $(d_j)\subset \mathcal{C}_0[z_0]$ such that $d_j\to d$. By assertion $(a)$ we have \begin{align*}
	\max_{(\lambda,  e^*,l,  w^*)\in\Lambda_*[z_0]}D_z^2\mathcal{L}(z_0, \lambda, e^*, l,  w^*)(d_j, d_j)\geq 0.
\end{align*} Passing to the limit, we obtain 
\begin{align*}
	\max_{(\lambda, e^*, l, w^*)\in\Lambda_*[z_0]}D_z^2\mathcal{L}(z_0, \lambda, e^*,l,  w^*)(d, d)\geq 0.
\end{align*}The proof of Theorem \ref{Theorem-SONC-VOP} is thereby completd. 
\end{proof}

\begin{remark}{\rm  In comparison with the obtained results in \cite{Bental}, Theorem \ref{Theorem-SONC-VOP} is an extension of \cite[Theorem 8.2]{Bental}. In \cite{Bental} the author considered problem {\rm (VOP)} without constraint $H(z)\leq 0$ and required that $K$ is a cone.  Here we consider problem {\rm (VOP)} with constraint $H(z)\leq 0$ and $K$ is a  convex set. Besides, our techniques for the proof of Theorem \ref{Theorem-SONC-VOP} is very different from those in \cite{Bental}. 
}	
\end{remark}

\section{Main results}

Let  $W^{1,p}([0, 1], \mathbb{R}^n)$ with $p\geq 1$ be a  Banach space of absolutely continuous vector-valued functions $x \colon [0, 1]\to\mathbb{R}^n$ such that $\dot x\in L^p([0,1], \mathbb{R}^n)$ with the norm $\|x\|_{1,p}:=|x(0)| +\|\dot x\|_p$, where  for each $\xi=[\xi_1, \ldots, \xi_n]^T\in\mathbb{R}^n$ $|\xi|=\left(\sum_{i=1}^{n}\xi_i^2\right)^{\frac{1}{2}}$ is the Euclidean norm of $\xi$. The symbols  $\|\cdot\|_p$ and $\|\cdot\|_0$  denote the norm in  $L^p([0, 1], \mathbb{R}^n)$ and in $C([0,1], \mathbb{R}^n)$, respectively. 

 Throughout this section, we assume that $\Phi$ is the feasible set of $P$ and let  
 \begin{align}
 	&X= C([0,1], \mathbb{R}^n),\ U=L^\infty([0,1], \mathbb{R}^m),\label{XU-spaces}\\
 	&K=\{v\in L^\infty([0,1], \mathbb{R}^r)| v(t)\leq 0\ {\rm a.a.}\ t\in [0, 1]\}.\label{K-set}
 \end{align}

Let $\psi$ stands for $L, \varphi$ and $g$, where $L=(L_1,..., L_k)$, $g=(g_1,..., g_r)$ and $\varphi=(\varphi_1,..., \varphi_n)$. Let $\Omega$ be an open set in $\mathbb{R}^n\times\mathbb{R}^m$ such that $(\bar x(t), \bar u(t))\in \Omega$ for all $t\in [0,1]$.	Given $(\bar x, \bar u)\in\Phi,$ the symbols  $L[t], \varphi[t], g[t], L_x[t], \varphi_u[t], g_u[t]$
and so on,  stand for
\begin{align*}
	&L(t, \bar x(t), \bar u(t)),\  \varphi(t, \bar x(t),\ \bar u(t)),
	g(t, \bar x(t), \bar u(t)),\\
	&L_x(t, \bar x(t), \bar u(t)),\  \varphi_u(t, \bar x(t),\ \bar
	u(t)), g_u(t, \bar x(t), \bar u(t)), \ldots
\end{align*} 
$$
	L(t, \bar x(t), \bar u(t)),\;  \varphi(t, \bar x(t), \bar u(t)),\;
	g(t, \bar x(t), \bar u(t)),\;
	L_x(t, \bar x(t), \bar u(t)), \varphi_u(t, \bar x(t), \bar
	u(t)), g_u(t, \bar x(t), \bar u(t)), \ldots
$$
 We shall denote by  $\ell'(\bar x(1))$ and $h'(\bar x(1))$ the first-order  derivatives of $\ell$ and $h$ at $\bar x(1)$, respectively. Also, by   $\ell''(\bar x(1))$ and $h''(\bar x(1))$ the second-order  derivatives of $\ell$ and $h$ at $\bar x(1)$, respectively. 
We impose the  following  assumptions.
\begin{enumerate}
	\item [$(H1)$] The  function $\psi$ is a Carath\'{e}odory function on $[0,1]\times \Omega$	and  $\psi(t,
	\cdot, \cdot)$ is of class $C^2$ on $\Omega$ for a.e. $t\in[0,1]$.  Besides, for each $M>0$,
	there exist numbers $k_{\psi K}>0$  such that
	\begin{align*}
		&|\psi(t, x_1, u_1)-\psi(t, x_2, u_2)|+|\nabla_z \psi(t, x_1, u_1)-\nabla_z \psi(t,
		x_2, u_2)|+\\
		&+|\nabla_z^2 \psi(t, x_1, u_1)-\nabla_z^2 \psi(t, x_2, u_2)|\leq
		k_{\psi M}(|x_1-x_2| +|u_1-u_2|)
	\end{align*} for all $(x_i, u_i)\in \Omega$  and $t\in [0,1]$ satisfying $|x_i|, |u_i|\leq M$ with $i=1,2$ and $z=(x,u)\in Z$.\\
		Moreover, we require that the functions
	$
	\psi(t, 0,0), |\nabla_z \psi (t, 0, 0)|, \left|\nabla_z^2 \psi(t, 0, 0)\right|
	$ belong to $L^\infty([0,1], \mathbb{R})$.
	
	\item[$(H2)$] The mappings $h(\cdot)$ and $\ell(\cdot )$ are of class $C^2$ around $\bar x(1)$ and ${\rm det} h'(\bar x(1))\neq 0$. 	
	
		\item [$(H3)$] There exists a number $\gamma>0$  such that 
	\begin{align} 
		|{\rm det} R[t]|\geq \gamma \quad \text{for a.a.}\ t\in [0, 1],
	\end{align}	where matrix $R[t]$ is defined by setting
\begin{align}
	R[t]= g_u[t]g_u[t]^T.
\end{align}
\end{enumerate}

Among the above assumptions, $(H1)$ and $(H2)$ make sure that $\psi$  is of class $C^2$ around $(\bar x, \bar u)$ while $(H3)$ together with rank of $g_u[t]$ play an important role for normality of multipliers. To derive KKT optimality conditions for ${\rm (P)}$, we will consider two cases:  ${\rm rank}(g_u[\cdot])<m$ and   ${\rm rank}(g_u[t])=m$. 

\medskip

\begin{enumerate}
	\item [$(H4)$] for a.a. $t\in[0, 1]$, ${\rm rank}(g_u[t])<m$ and ${\rm dim} (T[t])= m_*$ with the basis $\{w_1,..., w_{m_*}\}\subset L^\infty([0,1], \mathbb{R}^m)$,  where the subspace $T[t]:=\{w\in \mathbb{R}^m| g_u[t]w=0\}$. 	
\end{enumerate} 

Let us define  matrices
\begin{align}
	&S[t]=[w_1, w_2,..., w_{m_*}],\label{S-matrix}\\
	&A[t]=\varphi_x[t]-\varphi_u[t]g_u[t]^T R[t]^{-1}g_x[t],\label{A-matrix}\\ 
	&B[t]=\varphi_u[t]S[t]\label{B-matrix}
\end{align}  and require that the linear equation $\dot x=A[t]x + B[t]u$ is controllable. Namely, we need the assumption
\begin{enumerate}	
	\item[$(H5)$] The mapping $H: L^\infty([0,1], \mathbb{R}^{m_*})\to \mathbb{R}^n$ is surjective, where 
	\begin{align}
		H(v):=\int_0^1\Omega(1, s)B[s]v(s)ds
	\end{align} with $\Omega(t,\tau)$ is the principal matrix solution to the equation $\dot x=A[t]x$ (see \cite[p.123]{Alekseev}). 
\end{enumerate}	
When ${\rm rank}(g_u[t])=m$, then $T[t]=\{0\}$ and $S=0$. Hence  $B[t]=0$ and the system $\dot x=A[t]x + B[t]u$ is not controllable. In this situation, instead of $(H4)$ and $(H5)$ we need the following assumption.
\begin{enumerate}	
	\item[$(H4)'$] There exists $(\hat x, \hat u)\in X\times U$ such that the following conditions are fulfilled:
	
	\noindent $(a)$ $\hat x=\int_0^{(\cdot)} (\varphi_x[s]\hat x+ \varphi_u[s]\hat u )ds$;
	
	\noindent $(b)$ $h_i(\bar x(1)) + h_i'(\bar x(1))\hat x(1)<0$ for  $i=1,2,.., n$;
	
	\noindent $(c)$ ${\rm esssup}(g_j[t] +g_{jx}[t]\hat x + g_{ju}[t]\hat u) <0$ for $j=1,2,.., r$. 
		
\end{enumerate} 
Next we define a critical cone. For this we put $I_*=\{i\in\{1,2,..., n\}: h_i(\bar x(1))=0\}$ and denote by $\mathcal{C}_*[\bar z]$ the set of vectors $(x,u)\in X\times U$ satisfying the following conditions:
\begin{enumerate}
	\item [($b_1$)] $\ell'(\bar x(1)) x(1)+ \int_0^1 \left(L_x[t]x(t) +L_u[t]u(t)\right)dt\in
	- \mathbb{R}^k_+$;
	\item [($b_2$)]
	$ x=\int_0^{(\cdot)}\big(\varphi_x[s]x+\varphi_u[s]u(s)\big)ds$;
	\item [($b_3$)] $h_i'(\bar x(1))x(1)\leq 0\ \text{for}\ i\in I_*$;
	\item[($b_4$)] $g_u[\cdot]x+ g_u[\cdot]u\in {\rm cone}(K-g[\cdot])$.		
\end{enumerate}  The closure of $\mathcal{C}_*[\bar z]$ in $X\times U$ is called a {\it critical cone} to the ${\rm (P)}$ at $\bar z$ and denoted by  $\mathcal{C}_\infty[\bar z]$. Each vector $(x, u)\in \mathcal{C}_\infty[\bar z]$ is called a critical direction. \\
The following theorem  gives first-and second-order necessary optimality conditions of KKT-type for locally weak Pareto solutions to ${\rm (P)}$.

\begin{theorem}\label{Theorem-SOC-MCP1} Suppose that $\bar z$ is a locally weak Pareto solution of  ${\rm (P)}$ under which assumptions $(H1)-(H5)$ or $(H1)-(H3)$ and $(H4)'$ are satisfied. Then for each vector $(\tilde x, \tilde u)\in \mathcal{C}_\infty[\bar z]$,  there exist  a vector	$\lambda\in \mathbb{R}^k_+$ with $|\lambda|=1$, a vector  $l\in \mathbb{R}^n$,  an absolutely	continuous function $p\colon [0, 1]\to \mathbb{R}^n$ and a vector	function $\theta=(\theta_1, \theta_2,..., \theta_r) \in L^\infty([0, 1], \mathbb{R}^r)$ such that the following conditions are fulfilled:
	
	\noindent $(i)$ (the adjoint equation)
	\begin{equation}\notag
		\begin{cases}
			\dot{p}(t)=-\varphi_x[t]^Tp(t)+L_x[t]^T\lambda+ g_x[t]^T\theta(t),\quad \text{a.e.}\quad t\in [0,1]\\
			p(1)=-(\lambda^T\ell'(\bar x(1)) +l^T h'(\bar x(1)));
		\end{cases}
	\end{equation}
	
	\noindent $(ii)$ (the stationary condition in $u$)
	$$
	L_u[t]^T\lambda - \varphi_u[t]^T p(t) +g_u[t]^T\theta(t)=0 \ \ \text{a.e.} \ \ t\in [0, 1];
	$$
	
	\noindent $(iii)$ (complementary conditions) 
	\begin{align}
			& l_i\geq 0,\  l_i h_i(\bar x(1))=0,\ i=1,2,..., n,\label{l-cond}\\
			&\theta_j(t)\geq 0,\ \theta_j(t) g_j[t]=0\  {\rm a.a.}\ t\in [0,1],\ j=1,2,..., r; \label{theta-cond}
	\end{align} 
	
	\noindent $(iv)$ (the nonnegative second-order condition) 
	\begin{align*}
	&\lambda^T\ell''(\bar x(1))\tilde x(1)^2 + \int_0^1\lambda^T[L_{xx}[t]\tilde x^2 +2L_{ux}[t]\tilde x\tilde u+L_{uu}[t]\tilde u^2] dt+l^T h''(\bar x(1))\tilde x(1)^2 \notag \\
	&-\int_0^1p^T[\varphi_{xx}[t]\tilde x^2 +2\varphi_{ux}[t]\tilde x\tilde u +\varphi_{uu}[t]\tilde u^2 ]dt +\int_0^1\theta^T[g_{xx}[t]\tilde x^2 +2g_{ux}[t]\tilde x\tilde u+ g_{uu}[t]\tilde u]dt\geq 0.
	\end{align*} 	
\end{theorem}
\begin{proof} We first formulate ${\rm (P)}$ in the form of ${\rm (VOP)}$. Put 
	\begin{align*}
	Z=X\times U,\ E= C([0,1], \mathbb{R}^n),\  W= L^\infty([0,1], \mathbb{R}^r)
	\end{align*} and define  mappings $F: Z\to E$, $H:Z\to\mathbb{R}^n$, $H:Z\to\mathbb{R}^n$ and $G: Z\to W$   by setting
	\begin{align}
		& F(x, u):=x-x_0-\int_0^{(\cdot)}\varphi(s, x(s), u(s))ds\label{F-map}\\
		& H(x, u):= h(x(1))\label{H-map}\mbox{ and } G(x, u):=g(\cdot, x, u).
	\end{align}
	Then ${\rm (P)}$ can be written in the following form of ${\rm (VOP)}$:
	\begin{equation*} {\rm {\rm (P)}}\quad \quad \quad
		\begin{cases}
			J(x, u)\to\min,\\
			F(x, u)=0,\;H(x, u)\leq 0 \mbox{ and } G(x, u)\in K.			
		\end{cases}
	\end{equation*}  We shall  show that ${\rm (P)}$ satisfies assumptions $(A1)$, $(A2)$ and $(A4)$ or $(A1), (A2)$ and $(A5)$ of Theorem \ref{Theorem-SONC-VOP}. By \eqref{K-set},  we have ${\rm int}(K)\neq\emptyset$. Hence $(A1)$ is valid.  By $(H1)$ and $(H2)$, $J$, $F$,  $G$ and $H$ are of class $C^2$ around $(\bar x, \bar u)$ and $\bar x(1)$, respectively. Hence $(A2)$ is fulfilled. By a simple calculation,  we have 
	\begin{align}\label{F-Derivative}
	&DF(\bar x, \bar u)[(x,u)]=\Big(x-\int_0^{(\cdot)}(\varphi_x[s]x(s)+\varphi_u[s]u(s))ds\Big),\\
	&DH(\bar z)[(x, u)]= h'(\bar x(1)) x(1) \mbox{ and }  DG(\bar x, \bar u)[(x, u)]=g_x[t]x(t)+ g_{u}[t]u(t).
	\end{align}  The following lemma shows that ${\rm (P)}$ satisfies $(A4)$.
	
\begin{lemma}\label{LemmaControlable1} Under assumptions $(H4)$ and $(H5)$, $(\bar x, \bar u)$ satisfies condition $(A4)$.	
\end{lemma}
	
\noindent {\it Proof}.  To verify  condition \eqref{RobinsonCond3} we prove that for any $e\in E$, $\xi\in\mathbb{R}^n$ and $w\in W$,  there exist $(x, u)\in Z$, $\eta\in {\rm cone}((-\infty, 0]^n-H(z_0))$  and $\zeta\in{\rm cone}(K- g[\cdot])$ such that 

		\begin{equation}
			DF(\bar x, \bar u)[(x, u)]=e,\quad
			DH(z_0)[(x, u)] -\eta =\xi,\quad
		DG(\bar x, \bar y)[(x, u)]- \zeta=w.	
		\end{equation}  
		%%%
		
		This equation is equivalent to the system:
		\begin{align}
			& x-\int_0^{(\cdot)}(\varphi_x[s]x(s)+\varphi_u[s]u(s))ds=e,\label{Lin1}\\
			& h'(\bar x(1)) x(1)- \eta= \xi \mbox{ and } g_x[t]x + g_u[t]u-\zeta =w.	\label{Lin3}
		\end{align} Note that $g_x[t]$ is a matrix with size $r\times n$, $g_u[t]$ is a matrix with size $r\times m$ and $S[t]$ is a matrix with size $m\times m_*$, which is given by \eqref{S-matrix}.   We take $\eta=0, \zeta=0$ and find $u$ in the form $u=S[t]v+ g_u[t]^T R^{-1}[t] \hat u$, where $\hat u\in L^\infty([0,1], \mathbb{R}^r)$ and $v\in L^\infty([0,1], \mathbb{R}^{m_*})$. Then the system \eqref{Lin1}-\eqref{Lin3} becomes

		\begin{align}
			&x=\int_0^{(\cdot)}\big(\varphi_x[s]-\varphi_u[s]g_u[s]^T R^{-1}[s]g_x[s])x(s) +\varphi_u[s]S[s] v(s)\big)ds
			 + \int_0^{(\cdot)}\varphi_u[s]g_u[s]^T R^{-1}[s] w(s)ds+ e\label{Onto1} \\
			& x(1)= h'(x(1))^{-1} \xi \mbox{ and }  \hat u =w-g_x[\cdot]x.	\label{Onto3}
		\end{align} 
%%%%%%		
		Here we used the fact that $g_u[t]S[t]v=0$.  The above system  can be written in the form
		\begin{align}
			& x=\int_0^{(\cdot)}\big(A[s]x(s)+ B[s]v(s)\big)ds\label{Onto-x1} +\int_0^{(\cdot)}\varphi_u[s]g_u[s]^T R^{-1}[s] w(s)ds+ e,\\
			& x(1)= h'(\bar x(1))^{-1}\xi \mbox{ and }  \hat u =w-g_x[\cdot]x, \label{Onto-x3}
		\end{align} where  $A$ and $B$ are defined by \eqref{A-matrix} and \eqref{B-matrix}, respectively. By setting $y=x-e$, the system becomes
		\begin{align}
			& y=\int_0^{(\cdot)}\big(A[s]y(s)+ B[s]v(s)\big)ds +\int_0^{(\cdot)}\big(A[s]e(s)+\varphi_u[s]g_u[s]^T R^{-1}[s] w(s)\big)ds,\label{Onto-y1}\\
			& y(1)= h'(\bar x(1))^{-1}\xi-e(1) \mbox{ and } \hat u =w-g_x[\cdot]x.	\label{Onto3}
		\end{align} Let $\Omega(t, \tau)$ be the principal matrix solution of the system	$\dot y =A y.$  Then for each $v\in L^\infty([0,1], \mathbb{R}^{m_*})$, the solution of equations \eqref{Onto-y1} is given by
		\begin{align}
			y(t)=\int_0^t\big(\Omega(t, s)B[s]v(s)+ \Omega(t,s)(A[s]e(s)+\varphi_u[s]g_u[s]^T R^{-1}[s] w(s))\big)ds.
		\end{align} By $(H4)$, the mapping 
		$$
		H(v)=\int_0^1\Omega(1, s)B[s]v(s)ds
		$$ is surjective. Hence, there exists $\hat v\in L^\infty([0,1], \mathbb{R}^{m_*})$ such that 

		$$
			h'(\bar x(1))^{-1}\eta-e(1)=
			\int_0^1\Omega(1, s)B[s]\hat v(s)ds +\int_0^1\Omega(1,s)(A[s]e(s)+\varphi_u[s]g_u[s]^T R^{-1}[s] w(s))ds. 
		$$
		Let us define 
		$$
		\hat y(t)=\int_0^t\Omega(t, s)B[s]\hat v(s)ds +\int_0^t\Omega(t,s)(A[s]e(s)+\varphi_u[s]g_u[s]^T R^{-1}[s] w(s))ds. 
		$$ Then $\hat y$ is a solution of \eqref{Onto-y1} and \eqref{Onto3} corresponding to $\hat v$.  Let $\tilde x= \hat y+ e$. Then $\tilde x$ is a solution of \eqref{Onto-x1}-\eqref{Onto-x3}  corresponding to $\hat v$. Set $\hat u =w-g_x[\cdot]\tilde x$ and $\tilde u=S[t]\hat v+ g_u[t]^T M^{-1}[t] \hat u$. Then $(\tilde x, \tilde u)$ and $(\eta, \zeta)=(0,0)$ satisfies \eqref{Lin1}-\eqref{Lin3}. The lemma is proved. 
		
\begin{lemma} Under $(H4)'$, $(\bar x, \bar u)$ satisfies $(A5)$.	
\end{lemma}

\noindent {\it Proof.}  It is easy to show that $DF(\bar x, \bar u)[X\times U]=E$. By $(a)$ of $(H4)'$, we have $DF(\bar x, \bar u)[(\hat x, \hat u)]=0$. By $(b)$ of $(H4)'$, we have $H_i(\bar x, \bar u) + DH_i(\bar x, \bar u)[(\hat x, \hat u)]<0$. By \eqref{K-set}, we have 
$$
{\rm int}(K)=\{(v_1,..., v_r)\in L^\infty([0,1], \mathbb{R}): {\rm esssup} v_j<0, j=1,2,.., r.\}
$$ Therefore, from $(c)$ in $(H4)'$, we have 
$$
G(\bar x, \bar u) + DG(\bar x, \bar u)[(\hat x, \hat u)]\in {\rm int}(K).
$$ Hence $(A5)$ is valid. The lemma is proved. 
	
\medskip 
	
	We now derive first and second-order optimality conditions of KKT-type  for ${\rm (P)}$. 	Let 
	$$
	\mathcal{L}(z, \lambda, l, v^*, w^*)=\lambda^T J(x, u)  + v^* F(z)+ l^TH(z)+ w^*G(z)
	$$ be the Lagrangian associated with the ${\rm (P)}$, where $z=(x, u)\in Z$.   According to $(b)$ of Theorem \ref{Theorem-SONC-VOP}, 
	for each critical direction $\tilde z=(\tilde x, \tilde u)\in
	\mathcal{C}_\infty[\bar z]$, there exist multipliers $\lambda\in
	\mathbb{R}^k_+$ with $|\lambda|=1$, $l\in\mathbb{R}^n$,  $v^*\in C([0, 1], \mathbb{R}^n)^*$ and $ w^*\in L^\infty([0,1], \mathbb{R}^r)^*$
	such that the following conditions are valid:
	\begin{align}
		&\lambda^T J_x(\bar x, \bar u)+  v^* F_x(\bar z)+ l^T H_x(\bar z) +w^* g_x[\cdot]=0,\label{NFC1}\\
		&\lambda^T J_u(\bar x, \bar u) + v^* F_u(\bar z) +  w^* g_u[\cdot]=0,\label{NFC2}\\
		& l_i\geq 0,\  l_i H_i(\bar z)=0, i=1,2,.., n, \label{NFC2+}\\
		& w^* \in N(K, G(\bar z)), \label{NFC3}\\
		& \lambda^T \nabla^2 J(\bar z)[\tilde z, \tilde z] + v^*
		D^2 F(\bar z)[\tilde z, \tilde z]+ l^T D^2H(\bar z)[\tilde z, \tilde z] +  w^* D^2 G(\bar z)[\tilde z, \tilde z] \geq 0.\label{NFC4}
	\end{align} Here $v^*$ is a signed Radon measure and $w^*$ is a
	signed and finite additive measure on $[0, 1]$ which is absolutely
	continuous w.r.t the Lebesgue measure $|\cdot|$ on $[0, 1]$. By
	Riesz's Representation (see \cite[Chapter 01, p. 19]{Ioffe} and
	\cite[Theorem 3.8, p. 73]{Hirsch}), there exists a vector
	function of bounded variation $\nu$, which is continuous from the
	right and vanish at zero such that
	$$
	\langle v^*, y\rangle =\int_0^1 y(t)^Td\nu(t)\ \ \forall y\in E,
	$$ where $\int_0^1 y(t)d\nu(t)$ is the Riemann-Stieltjes integral.
	
	Define $\bar p \colon [0, 1]\to \mathbb{R}^n$ by setting
	$$
	\bar p(t) =\nu((t, 1])=\int_t^1 d\nu(s)=\nu(1)-\nu(t).
	$$ Clearly, $\bar p(1)=0$ and the function $\bar p$ is of bounded variation. 
	
	By \cite[Theorem 2.21]{Wheeden},  for any  $a\in C([0,1], \mathbb{R}^n)$, we have 
	\begin{align}\label{StiejetInbyPart}
		\int_0^1 a(s)^Td\nu(s)=a(1)^T\nu(1)-a(0)^T\nu(0)-\int_0^1 \nu(s)da(s)^T.
	\end{align} Applying  formula \eqref{StiejetInbyPart} for $a(t)=\int_0^t \varphi_x[s]x(s)ds$, we get 
	\begin{align}
		\langle v^* F_x(\bar z), x\rangle
		&=\left\langle v^*,	x-\int_0^{(\cdot)}\varphi_x[s]x(s)ds\right\rangle\notag\\
		&=\int_0^1 x^T(t)d\nu(t)-\int_0^1\int_0^t (\varphi_x[s] x(s))^Tds d\nu(t)\notag\\
		&=\int_0^1 x^T(t)d\nu(t) -[\int_0^1 (\varphi_x[s] x(s))^T\nu(1)ds- \int_0^1(\varphi_x[s]x(s))^T\nu(s)ds] \notag\\
		&=\int_0^1 x^T(t)d\nu(t) - \int_0^1
		x^T(t)\varphi_x[t]^T \bar p(t)dt. \notag\\
		&=\int_0^1 x^T(t)d\nu(t) - \int_0^1
		\bar p^T(t)\varphi_x[t]x(t)dt.\label{Action1}
	\end{align} 
	Similarly, for any $u\in L^\infty([0, 1], \mathbb{R}^m)$, we get
	\begin{equation}\label{Action2}
		\left\langle v^* F_u(\bar z), u\right\rangle=-\int_0^1\int_0^t(\varphi_u[s]u(s))^Tds
		d\nu(t)=-\int_0^1 u(s)^T \varphi_u[s]^T \bar p(s) ds
	\end{equation}  and 
	\begin{equation}\label{Action3}
		\left\langle v^*\nabla^2F(\bar z)z, z\right\rangle=-\int_0^1 \left(\bar
		p(t)^T\nabla^2\varphi[s]z(s), z(s)\right)ds.
	\end{equation}From \eqref{NFC2} and \eqref{Action2}, we have
	\begin{equation}\label{NFC5}
		\int_0^1\lambda^T L_u[s]u(s)ds -\int_0^1 \bar p(s)^T \varphi_u[s]
		u(s) ds +\langle w^*,g_u[\cdot]u\rangle=0\ \ \forall u\in U.
	\end{equation}
	Let us claim that $w^*$ can be represented by a density
	in $L^\infty([0, 1], \mathbb{R}^r)$. Indeed,  by the formula of inverse matrix (see \cite[Theorem 8.1, p. 176]{Lang}), we have 
	$$
	R^{-1}[t]=\big(\frac{e_{ij}(t)}{{\det R[t]}}\big),
	$$ where $e_{ij}(\alpha)$ are cofactors. By $(H3)$, $|\det R[t]|\geq\gamma_0$ for a.a $t\in[0,1]$. Hence   $R^{-1}[\cdot]\in L^\infty([0, 1], \mathbb{R}^{r\times r})$. Taking any $w\in L^\infty([0,1], \mathbb{R}^r)$, we put $u(t)=g_u[t]^T R^{-1}[t] w(t)$. Inserting $u$ into \eqref{NFC5}, we have 
	\begin{align}\label{Dens1}
		\int_0^1\lambda^T L_u[t]g_u[t]^T R^{-1}[t] w(t) dt - \int_0^1  \bar p^T(t) \varphi_u[t] g_u[t]^T R^{-1}[t] w(t) dt +\langle w^*, w\rangle=0.
	\end{align} This implies that 
	\begin{align*}
		\langle w^*, w\rangle= \int_0^1 (-\lambda^T L_u[t]g_u[t]^T R^{-1}[t] +\bar p^T(t)\varphi_u[t]g_u[t]^T R^{-1}[t])w(t)dt. 
	\end{align*} Let us define a function  
	$$
	\theta(t)= -\lambda^T L_u[t]g_u[t]^T R^{-1}[t] +\bar p^T(t)\varphi_u[t]g_u[t]^T R^{-1}[t].
	$$ Then $\theta \in L^\infty([0, 1], \mathbb{R})$ and 
	\begin{align}\label{RepDensity}
	\langle w^*, w\rangle=\int_0^1 \theta(t)w(t) dt\quad \forall w\in L^\infty([0,1], \mathbb{R}^r).  
	\end{align} The claim is justified.  Combining this with \eqref{NFC5}, we get 
	\begin{align*}
		\int_0^1\lambda^TL_u[t] u(t)dt-\int_0^1 \bar p(t)^T\varphi_u[t]u(t)dt + \int_0^1\theta(t)^Tg_u[t]u(t) dt =0\quad \forall u\in L^\infty([0,1], \mathbb{R}^m).
	\end{align*} It follows that 
\begin{align}
\lambda^TL_u[t]- \bar p(t)^T\varphi_u[t] +\theta(t)^T g_u[t]	=0 \quad \text{a.a.}\quad t\in [0,1]. 
\end{align} We obtain assertion $(ii)$ of the theorem. 
 Let us define 
	\begin{align} 
		p(t)=
		\begin{cases}
			\bar p(t)\quad &\text{if}\quad t\in [0, 1)\\
			\displaystyle\lim_{t\to 1^-}\bar p(t) \quad &\text{if}\quad t=1.
		\end{cases}
	\end{align} Note that although $\nu$ is right continuous, it may not be continuous at $1$ and so is $\bar p$. 
	
	We now have from \eqref{NFC1} and  \eqref{Action1} that 
	\begin{align*}
		&\int_0^1 x^T(s)d\nu(s) - \int_0^1\bar p^T(s)\varphi_x[s]x(s)ds +\int_0^1\lambda^T L_x[s]x(s)ds +\int_0^1 \theta(s)^T g_x[s]ds\\
		& +\lambda^T\ell_{x_1}(\bar x(1))x(1) + l^Th_{x_1}(\bar x(1)) x(1) =0\quad  \forall x\in X.
	\end{align*} This is equivalent to 
	\begin{align*}
		&\int_0^1 x^T(s)d\nu(s) = \int_0^1\bar p^T(s)\varphi_x[s]x(s)ds -\int_0^1\lambda^T L_x[s]x(s)dt -\int_0^1 \theta(s)^T g_x[s] x(s)ds\\
		& -\lambda^T\ell_{x_1}(\bar x(1))x(1) - l^Th_{x_1}(\bar x(1)) x(1)\quad    \forall x\in X.
	\end{align*} By definition of $p$, we have 
	\begin{align}\label{AdjEq3}
		&\int_0^1 x^T(s)d\nu(s) = \int_0^1 p^T(s)\varphi_x[s]x(s)ds -\int_0^1\lambda^T L_x[s]x(s)ds -\int_0^1 \theta(s)^T g_x[s] x(s)ds\notag\\
		& -\lambda^T\ell_{x_1}(\bar x(1))x(1) - l^Th_{x_1}(\bar x(1)) x(1)\quad    \forall x\in X.
	\end{align} By Lemma 5.1 in \cite{Kien3}, the above equality is also valid for $x_t(s)=\xi\chi_{(t, 1]}(s)$ with $0 \leq t <1$, where $\xi\in\mathbb{R}^n$ and $\chi_{(t, 1]}$ is the indicator function of the set $(t, 1]\subset [0,1]$.  Inserting $x_t(s)$ into \eqref{AdjEq3}, we get 
	\begin{align*}
		\int_t^1 \xi^T d\nu(s) = \int_t^1 \xi^T(p^T(s)\varphi_x[s] -\lambda^T L_x[s]- \theta(s)^T g_x[s])ds -\xi^T\big(\lambda^T\ell_{x_1}(\bar x(1)) - l^Th_{x_1}(\bar x(1))\big)	
	\end{align*} for all $t\in [0, 1)$.  Since $\xi$ is arbitrary, we obtain
	\begin{align*}
		\int_t^1 d\nu(s) = \int_t^1 (p^T(s)\varphi_x[s] -\lambda^T L_x[s]- \theta(s)^T g_x[s])ds -(\lambda^T\ell_{x_1}(\bar x(1)) + l^Th_{x_1}(\bar x(1)).
	\end{align*} This means 
	\begin{align}\label{AdjEq4}
		\bar p(t) = \int_t^1 (p^T(s)\varphi_x[s] -\lambda^T L_x[s]- \theta(s)^T g_x[s])ds -(\lambda^T\ell_{x_1}(\bar x(1)) + l^Th_{x_1}(\bar x(1))
	\end{align} for all $t\in [0, 1)$. It follows that 
	\begin{align}
		\dot p(t) =  -p^T(t)\varphi_x[t] +\lambda^T L_x[t]+\theta(t)^T g_x[t]\quad \text{a.a.}\quad t\in (0, 1). 
	\end{align} Taking the limit both sides of \eqref{AdjEq4} when $t\to 1^{\text -}$, we obtain 
	$$
	p(1)=\lim_{t\to 1^{\text{-}}}\bar p(t)=-(\lambda^T\ell_{x_1}(\bar x(1)) + l^Th_{x_1}(\bar x(1)).
	$$ Hence assertion $(i)$ is followed. Since $H_i(\bar z)=h_i(\bar x(1))$, condition \eqref{NFC2+} implies that  $l_i h_i(\bar x(1))=0$. 	 From condition \eqref{NFC3}, we have 
	$$
	\int_0^1 \theta^T(w(t)-g[t])dt=\langle w^*, w-G(\bar z)\rangle \leq 0\quad \forall w\in K. 
	$$ This and \cite[Corolary 4.4]{Pales2} imply that $\theta(t)\in N((-\infty, 0]^r, g[t])$ for a.e. $t\in [0,1]$. Hence $\theta_j(t)g_j[t]=0$  and $\theta_j(t)\geq 0$ for a.a. $t\in [0, 1]$ and for all $j=1,2,..., r$.  We obtain \eqref{theta-cond} in  assertion $(iii)$.
	
	 Finally, by combining \eqref{NFC4}, \eqref{Action3} and \eqref{RepDensity}, we derive assertion $(iv)$.   The proof of the theorem is complete. 	
\end{proof}

\medskip

\noindent To derive  sufficient conditions, we  need to enlarge the critical cone. For this we define $\mathcal{C}_2[\bar z]$  which consists of couples $(x, u)\in X\times L^2([0,1], \mathbb{R}^m)$ such that 
\begin{enumerate}
	\item [($b'_1$)] $\ell'(\bar x(1)) x(1)+ \int_0^1 \left(L_x[t]x(t) +L_u[t]u(t)\right)dt\in
	- \mathbb{R}^k_+$;
	\item [($b'_2$)]
	$ x=\int_0^{(\cdot)}\big(\varphi_x[s]x+\varphi_u[s]u(s)\big)ds$;
	\item [($b'_3$)] $h_i'(\bar x(1))x(1)\leq 0\quad \text{for}\quad i\in I_*$; 	
	\item[($b'_4$)] $g_u[\cdot]x+ g_u[\cdot]u\in \overline{\rm cone}(K- g[\cdot])$.		
\end{enumerate} Clearly, $\mathcal{C}_\infty[\bar z]\subseteq \mathcal{C}_2[\bar z]$.

\begin{definition}
	Let $\bar{z}=(\bar{x},\bar{u}) \in \Phi$ be a feasible point of ${\rm (P)}$. We say that $\bar{z}$ is a locally strong Pareto solution of ${\rm (P)}$ if there exist a number $\epsilon>0$ and a vector $c\in {\rm int} \mathbb{R}^k_+$ such that for all $( x, u)\in (B_X(\bar x,\epsilon)\times B_U(\bar u, \epsilon))\cap \Phi$, one has
	$$
	J(x,u)-J(\bar x,\bar u) - c\|u-\bar u\|_2^2\notin -{\rm int}(\mathbb{R}^k_+).
	$$
\end{definition}
It is easy to show that every locally strong Pareto solution of ${\rm (P)}$ is also a locally Pareto solution of this problem.  The following theorem provides sufficient conditions for locally strong Pareto solutions to ${\rm (P)}$. 

\begin{theorem}\label{Theorem2-SOSC}
	Suppose  $(\bar x, \bar u)\in \Phi$, assumptions $(H1)-(H3)$,    multipliers  $\lambda\in \mathbb{R}^k_+$ with $|\lambda|=1$, $l\in\mathbb{R}^n$, the absolutely continuous function $p:[0,1]\to \mathbb{R}^n$ and a function $\theta \in L^\infty([0,1],\mathbb{R}^r)$  such that  conditions $(i)-(iii)$  of Theorem 3.1 are fulfilled.  Furthermore, assume that 
	
	\noindent $(iv)'$ (the strictly second-order condition)
	\begin{align*}
	&\lambda^T\ell''(\bar x(1))\tilde x(1)^2 + \int_0^1\lambda^T[L_{xx}[t]\tilde x^2 +2L_{ux}[t]\tilde x\tilde u+L_{uu}[t]\tilde u^2] dt+l^T h''(\bar x(1))\tilde x(1)^2 \notag \\
	&-\int_0^1p^T[\varphi_{xx}[t]\tilde x^2 +2\varphi_{ux}[t]\tilde x\tilde u +\varphi_{uu}[t]\tilde u^2 ]dt +\int_0^1\theta^T[g_{xx}[t]\tilde x^2 +2g_{ux}[t]\tilde x\tilde u+ g_{uu}[t]\tilde u]dt> 0
\end{align*} for all $(\tilde x, \tilde u)\in \mathcal{C}_2[\bar z]\setminus\{0\}$; 
	
	\noindent $(v)'$ there exists a number $\gamma_0>0$ such that for a.a. $t\in [0,1]$ one has
	$$
	\lambda^T L_{uu}[t](v, v)\geq \gamma_0|v|^2\quad \forall v\in \mathbb{R}^m.  
	$$		
	Then $(\bar x,\bar u)$ is a locally strong Pareto solution of $( P)$.
\end{theorem}
\begin{proof}  Suppose the conclusion of  the theorem was
	false. Then, we could find sequences $\{(x_j, u_j)\}\subset  \Phi$ and
	$\{c_j\}\subset \mathrm{int}\,\mathbb{R}^k_+$ such that $(x_j, u_j) \to (\bar x,\bar u)$ as $j\to\infty$,  $c_j\to 0$ and
	\begin{equation}\label{Inq1}
		J(x_j, u_j)- J(\bar x, \bar u) -c_j\|u_j-\bar u\|_2^2\in
		-{\rm int}\mathbb{R}^k_+.
	\end{equation} It follows that $(x_j, u_j)\neq (\bar x, \bar u)$ for all $j\geq 1$. By a simple argument, we can show that if $u_j= \bar u$ then $x_j=\bar x$,  which is absurd. Hence we must have $u_j\neq \bar u$ for all $j\geq 1$. Let 
$$
\mathcal{L}(z, \lambda,  p, l, \theta):=\lambda^T J(z) +  p^TF(z)+  l^T H(z) +\theta^T G(z)
$$ be the Lagrange function associated with ${\rm (P)}$, where
\begin{align*}
	&\lambda^T J(z)=\lambda^T\ell(x(1))+ \int_0^1\lambda^T L(s, x(s), u(s))ds,\\
	& p^T F(z)=-\int_0^1 \dot p^T(s)x(s)ds -\int_0^1 p^T(s) 
	\varphi(s, x(s), u(s))ds + p^T(1) x(1),\\
	& l^T H(z)= l^T h(x(1)),\\
	&\theta^T G(z)=\int_0^1 \theta^T(s) g(s, x(s), u(s))ds.
\end{align*}
Then, from conditions $(i)$ $(ii)$ and $(iii)$ of Theorem 3.1, we can show that
\begin{equation}\label{FCond}
	\nabla_z \mathcal{L}(\bar z, \lambda,  p, l, \theta)=0.
\end{equation}
Define $t_j=\|u_j- \bar u\|_2$, $\hat x_j= \frac{x_j-\bar x}{t_j}$ and   $\hat u_j= \frac{u_j-\bar u}{t_j}$. Then $t_j\to 0^+$ and $\|\hat u_j \|_2=1$. By reflexivity of  $L^2([0, 1],
\mathbb{R}^m)$,   we may assume that $\hat u_j
\rightharpoonup \hat u$. Since $\lambda\in\mathbb{R}^k_+$, we have from \eqref{Inq1} that
\begin{equation}\label{KeyInq2}
	\lambda^T J(z_j)-\lambda^T J(\bar z)\leq t_j^2\lambda^T c_j\leq
	t_j^2 |\lambda||c_j|\leq o(t_j^2).
\end{equation} For clarification, we divide the remain of the  proof into some steps.

\medskip

\noindent {\it Step 1}. Showing hat $\hat x_j$ converges uniformly to some
$\hat x$ in $C([0,1], \mathbb{R}^n)$.

 Indeed, since $(\bar x, \bar u), (x_j,
u_j)\in\Phi$, we get 

$$x_j(t)= x_0+\int_0^t \varphi(s, x_j(s), u_j(s))ds,\quad \bar x(t)= x_0+\int_0^t \varphi(s, \bar x(s), \bar u(s))ds.
$$
Hence 
\begin{equation}\label{DifE1}
	t_j\hat x_j(t)=\int_0^t (\varphi(s, x_j(s), u_j(s))-\varphi(s, \bar
	x(s), \bar u(s))ds.
\end{equation} Since $x_j\to \bar x$ uniformly and $u_j\to \bar u$ in
$L^\infty([0, 1], \mathbb{R}^m)$, there exists a constant
$\varrho>0$ such that $\|x_j\|_0\leq\varrho,
\|u_j\|_\infty\leq\varrho$. By assumption $(H1)$, there exists
$k_{\varphi, \varrho}>0$ such that
$$
|\varphi(s, x_j(s), u_j(s))-\varphi(s, \bar x(s), \bar u(s))|\leq
k_{\varphi, \varrho}(|x_j(s)-\bar x(s)| +|u_j(s)-\bar u(s)|)
$$ for a.e. $s\in[0, 1]$. Hence we have from \eqref{DifE1} that

\begin{equation}\label{DEInq1}
	|\hat x_j(t)|\leq \int_0^t k_{\varphi, \varrho}(|\hat x_j(s)|+ |\hat
	u_j(s)|)ds
\mbox{ and }
	|\dot{\hat x}_j(t)|\leq  k_{\varphi, \varrho}(|\hat x_j(t)|+ |\hat
	u_j(t)|).
\end{equation} 

Since $\|\hat u_j\|_2=1$, we get
\begin{align*}
	|\hat x_j(t)|&\leq \int_0^t k_{\varphi, \varrho}|\hat x_j(s)|ds+
	\int_0^1 k_{\varphi, \varrho}|\hat u_j(s)|ds\\
	&\leq \int_0^tk_{\varphi, \varrho}|\hat x_j(s)|ds+ k_{\varphi,
		\varrho}\left(\int_0^1 |\hat u_j(s)|^2 ds\right)^{1/2}\\
	&\leq \int_0^t k_{\varphi, \varrho}|\hat x_j(s)|ds+ k_{\varphi,
		\varrho}.
\end{align*} By the Gronwall Inequality (see \cite[18.1.i, p. 503]{Cesari}), we have
\begin{equation*}
	|\hat x_j(t)|\leq k_{\varphi, \varrho}\exp(k_{\varphi, \varrho}).
\end{equation*}  Combining  this with 
\eqref{DEInq1} yields
\begin{equation*}
	|\dot{\hat x}_j(t)|^2\leq  2k^2_{\varphi, \varrho}\left(|\hat x_j(t)|^2+
	|\hat u_j(t)|^2\right)\leq  2k^2_{\varphi, \varrho}(k^2_{\varphi,
		\varrho}\exp(2k_{\varphi, \varrho}) + |\hat u_j|^2).
\end{equation*}
Hence
$$
\int_0^1 |\dot{\hat x}_j(t)|^2 dt\leq 2k^2_{\varphi,
	\varrho}(k^2_{\varphi, \varrho}\exp(2k_{\varphi, \varrho})+1).
$$ Consequently, $\{\hat x_j\}$ is bounded in $W^{1,2}([0, 1],
\mathbb{R}^n)$. By passing subsequence, we may assume that $\hat
x_j\rightharpoonup \hat x$ weakly in $W^{1,2}([0, 1],
\mathbb{R}^n)$. As  the embedding
$W^{1,2}([0, 1], \mathbb{R}^n)\hookrightarrow C([0, 1],
\mathbb{R}^n)$ is compact (see \cite[Theorem 8.8]{Brezis1}) we have $\hat x_j\to\hat x$ in norm of $C([0, 1], \mathbb{R}^n)$. 

\medskip

\noindent {\it Step 2}. Showing that $(\hat x, \hat u)\in\mathcal{C}_2[\bar z].$

Using the mean value theorem for the left hand side of \eqref{Inq1} and dividing both sides by $t_j$, we get 
\begin{align*}
&\ell'(\bar x(1)+ \gamma_j(x_j(1)-\bar x(1)))\hat x_j(1)+ \int_0^1 L_x(\bar x +\xi_j(x_j-\bar x), u_j)\hat x_j dt\\
& +\int_0^1 L_u(\bar x, \bar u + \xi_j(u_j-\bar u))\hat u_j dt
	+\frac{o(t^2_j)}{t_j}\in -{\rm int}(\mathbb{R}_+^k),
\end{align*} where $0\leq \gamma_j\leq 1$ and $0\leq \xi_j(t)\leq 1$. Letting $j\to\infty$ and using the dominated convergence theorem, we obtain
\begin{align}\label{CriticalCon1}
	\ell'(\bar x(1))\hat x(1) +\int_0^1 (L_x[t]\hat x(t) +L_u[t]\hat u(t))dt \in -\mathbb{R}^k_{+}. 
\end{align} Hence $(\hat x, \hat u)$ satisfies condition $(b_1')$. Also, from \eqref{DifE1} we have 
\begin{align*}
	\hat x_j(t)=\int_0^t\big(\varphi_x(\bar x +\xi_j(x_j-\bar x), u_j)\hat x_j +
	\varphi_u(\bar x, \bar u + \eta_j(u_j-\bar u))\hat u_j\big)dt,
\end{align*} where $0\leq \xi_j(s), \eta_j(s)\leq 1.$ Letting $j\to\infty$ and using the dominated convergence theorem, we obtain
$$
\hat x (t)=\int_0^t( \varphi_x[s]\hat x(s)+ \varphi_u[s]\hat u(s))ds.
$$ Hence $(\hat x, \hat u)$ satisfies condition $(b_2')$.  Obviously, we have
$$
h(x_j(1))-h(\bar x(1))\in \big((-\infty, 0]^n- h(\bar x(1))\big). 
$$ By a Taylor's expansion, we get 
\begin{align*}
	h'(\bar x(1)+ \tau(x_j(1)-\bar x(1)))\hat x_j(1)\in \frac{1}{t_j} \big((-\infty, 0]^n- h(\bar x(1))\big)\subset T((-\infty, 0]^n, h(\bar x(1))).
\end{align*} By letting $j\to\infty$, we obtain
$$
h'(\bar x(1))\hat x(1)\in T((-\infty, 0]^n, h(\bar x(1))).
$$ Hence  $(\hat x, \hat u)$ satisfies condition $(b_3')$. It remains to verify condition $(b_4')$. 

Since $G(x_j, u_j)-G(\bar x, \bar u)\in K-G(\bar x, \bar u)$, we have
\begin{align}\label{c5-CriCone}
\frac{1}{t_j}(G(x_j, u_j)-G(\bar x, \bar u))\in {\rm cone}(K-G(\bar z))\subset \overline{{\rm cone}}(K-G(\bar z)).
\end{align}  Note that
$$
\overline{{\rm cone}}(K-G(\bar z))=T(K; G(\bar z))=\left\{ v\in L^2([0, 1], \mathbb{R}^r) \;|\; v(t)\in T((-\infty, 0]^r, g[t])\ {\rm a.e.}\right\}.
$$ Using the mean value theorem, we have from \eqref{c5-CriCone} that 
\begin{align}\label{c5-CriCone2}
g_x(t, \bar x + \eta_j(x_j-\bar x), \bar u)\hat x_j + g_u(t, x_j,  \bar u + \zeta_j(u_j-\bar u))\hat u_j \in T(K; G(\bar z)),	
\end{align}  where $0\leq \eta_j(t), \zeta_j(t)\leq 1$ a.a. $t\in [0,1]$.
Since  $T(K; G(\bar z))$ is a closed
convex set in $L^2([0, 1], \mathbb{R}^r)$, it is also a weakly closed
set in $L^2([0, 1], \mathbb{R}^r)$. By $(H1)$, we can show that 
$$
g_x(\cdot, \bar x + \eta_j(x_j-\bar x), \bar u)\hat x_j\to g_x[\cdot]\hat x
$$ strongly in $L^2([0, 1], \mathbb{R}^r)$ and 
$$
g_u(t, x_j,  \bar u + \zeta_j(u_j-\bar u))\hat u_j\to g_u[\cdot]\hat u
$$ weakly in $L^2([0, 1], \mathbb{R}^r)$. Letting $j\to\infty$, we obtain from that \eqref{c5-CriCone2} that 
$$
g_x[\cdot]\hat x + g_u[\cdot]\hat u\in T(K, g[\cdot]).
$$ Hence $(\hat x, \hat u)$ satisfies condition $(b_4')$ and so $(\hat x, \hat u)\in\mathcal{C}_2[\bar z]$. 

\medskip

\noindent {\it Step 3}.  Showing that $(\hat x, \hat u)=(0,0)$.

By \eqref{l-cond}, we have $l^T(h(x_j(1))-h(\bar x(1)))\leq 0$. By \eqref{theta-cond},   $\int_0^1 \theta^T(t)(g(t, x_j, u_j)-g[t])dt\leq 0$. Besides, we have  
\begin{align*}
&\int_0^1\big(-\dot p^T(x_j-\bar x) -p(\varphi(t, x_j, u_j)-\varphi(t, \bar x, \bar u))\big)dt+ p^T(1)(x_j(1)-\bar x(1))\\
&=\int_0^1 p\big((\dot{x}_j -\dot{\bar x})-(\varphi(t, x_j, u_j)-\varphi(t, \bar x, \bar u))\big) dt=0.
\end{align*}
From the above, \eqref{KeyInq2} and definition of $\mathcal{L}$, we have 

%%%%%%%%%%%%%%
\begin{align*}
&\mathcal{L}(x_j, u_j \lambda,  p, l,  \theta)-\mathcal{L}(\bar x, \bar u, \lambda,  p, l,  \theta)	
= \lambda^T(\ell(x_j(1))-\ell(\bar x(1))) +\lambda^T\int_0^1(L(t, x_j, u_j)-L(t, \bar x, \bar u))dt \\
&+l^T(h(x_j(1))-h(\bar x(1)))-\int_0^1\big(\dot p^T(x_j-\bar x) +p(\varphi(t, x_j, u_j)-\varphi(t, \bar x, \bar u))\big)dt\\
&+ \int_0^1 \theta^T(t)(g(t, x_j, u_j)-g(t, \bar x, \bar u))dt\\
\leq&\lambda^T(\ell(x_j(1))-\ell(\bar x(1))) +\lambda^T\int_0^1(L(t, x_j, u_j)-L(t, \bar x, \bar u))dt\leq o(t_j^2).
\end{align*}
\allowdisplaybreaks
Using a second-order Taylor expansion for $\mathcal{L}$ and \eqref{FCond}, we get
\begin{align*}
o(t_j^2)&\geq \mathcal{L}(z_k, \lambda ,\bar p, l, \theta)-\mathcal{L}(\bar
	z,\lambda,\bar p, l, \theta)\\
     =& \frac{t_j^2}{2}[\lambda^T\ell''(\bar x(1)+\xi_{1j}(x_j(1)-\bar x(1)))\hat x_j(1)^2 + l^T h''(\bar x(1)+\xi_{2j}(x_j(1)-\bar x(1)))\hat x_j(1)^2]\\
     & +\frac{t_j^2}{2}\lambda^T\int_0^1[L_{xx}(t, \bar x+\eta_{1j}(x_j-\bar x), \bar u)\hat x_j^2 +2L_{ux}(t,\bar x+\eta_{2j}(x_j-\bar x), \bar u)\hat x_j\hat u_j]dt\\
     & +\frac{t_j^2}{2}\lambda^T\int_0^1L_{uu}(t, x_j, \bar u+\eta_{3j}(u_j-\bar u))\hat u_j^2 ]dt\\
     &-\frac{t_j^2}{2}\int_0^1p^T[\varphi_{xx}(t, \bar x+\eta_{4j}(x_j-\bar x), \bar u)\hat x_j^2 +2\varphi_{ux}(t,\bar x+\eta_{5j}(x_j-\bar x), \bar u)\hat x_j\hat u_j]dt\\
     & -\frac{t_j^2}{2}\int_0^1 p^T\varphi_{uu}(t, x_j, \bar u+\eta_{6j}(u_j-\bar u))\hat u_j^2 ]dt\\
    & +\frac{t_j^2}{2}\int_0^1\theta^T[g_{xx}(t, \bar x+\eta_{7j}(x_j-\bar x), \bar u)\hat x_j^2 +2g_{ux}(t,\bar x+\eta_{8j}(x_j-\bar x), \bar u)\hat x_j\hat u_j]dt\\
     & +\frac{t_j^2}{2}\int_0^1 \theta^T g_{uu}(t, x_j, \bar u+\eta_{9j}(u_j-\bar u))\hat u_j^2 ]dt,
\end{align*} where $0\leq \xi_{ij}, \eta_{ij}(t)\leq 1$. This implies that 
\begin{align}\label{KeyInq3}
	\frac{o(t_j^2)}{t_j^2}\geq&
	 \lambda^T\ell''(\bar x(1)+\xi_{1j}(x_j(1)-\bar x(1)))\hat x_j(1)^2 + l^T h''(\bar x(1)+\xi_{2j}(x_j(1)-\bar x(1)))\hat x_j(1)^2\notag\\
	& +\lambda^T\int_0^1[L_{xx}(t, \bar x+\eta_{1j}(x_j-\bar x), \bar u)\hat x_j^2 +2L_{ux}(t,\bar x+\eta_{2j}(x_j-\bar x), \bar u)\hat x_j\hat u_j]dt\notag\\
	& +\lambda^T\int_0^1L_{uu}(t, x_j, \bar u+\eta_{3j}(u_j-\bar u))\hat u_j^2 dt \notag \\
	&-\int_0^1p^T[\varphi_{xx}(t, \bar x+\eta_{4j}(x_j-\bar x), \bar u)\hat x_j^2 +2\varphi_{ux}(t,\bar x+\eta_{5j}(x_j-\bar x), \bar u)\hat x_j\hat u_j]dt\notag \\
	& -\int_0^1 p^T\varphi_{uu}(t, x_j, \bar u+\eta_{6j}(u_j-\bar u))\hat u_j^2 ]dt\notag\\
	& +\int_0^1\theta^T[g_{xx}(t, \bar x+\eta_{7j}(x_j-\bar x), \bar u)\hat x_j^2 +2g_{ux}(t,\bar x+\eta_{8j}(x_j-\bar x), \bar u)\hat x_j\hat u_j]dt\notag\\
	& +\int_0^1 \theta^T g_{uu}(t, x_j, \bar u+\eta_{9j}(u_j-\bar u))\hat u_j^2 ]dt.
\end{align} Since $\|u_j-\bar u\|_\infty\to 0$, $\|x_j-\bar x\|_0\to 0$ and $(H1)$, we have 
$$
\|\lambda^TL_{uu}(\cdot, x_j, \bar u+\eta_{3j}(u_j-\bar u))\hat u_j^2-\lambda^TL[\cdot]\|_\infty \to 0\quad {\rm as}\quad j\to\infty.
$$ By $(v)'$, the functional 
$$
u\mapsto \int_0^1 \lambda^T L_{uu}[t]u^2 dt
$$ is convex and sequentially lower semicontinuous. It follows that 

\begin{align}
	\lim_{j\to\infty} \lambda^T\int_0^1L_{uu}(t, x_j, \bar u+\eta_{3j}(u_j-\bar u))\hat u_j^2 dt=&\lim_{j\to\infty}\int_0^1(\lambda^TL_{uu}(t, x_j, \bar u+\eta_{3j}(u_j-\bar u))-\lambda^TL_{uu}[t])\hat u_j^2 dt\notag\\
	&+\lim_{j\to\infty}\int_0^1\lambda^T L_{uu}[t]\hat u_j^2 dt.\notag\\
	=&\lim_{j\to\infty}\int_0^1\lambda^T L_{uu}[t]\hat u_j^2 dt\geq \max(\gamma_0, \int_0^1\lambda^TL_{uu}[t]\hat u^2 dt). \label{KeyInq5}
\end{align} 
Using this fact and taking limits both sides of \eqref{KeyInq3}, we obtain
\begin{align*}
	&0\geq
	\lambda^T\ell''(\bar x(1))\hat x(1)^2 + l^T h''(\bar x(1))\hat x(1)^2+\int_0^1\lambda^T[L_{xx}[t]\hat x^2 +2L_{ux}[t]\hat x\hat u+L_{uu}[t]\hat u^2] dt \notag \\
	&-\int_0^1p^T[\varphi_{xx}[t]\hat x^2 +2\varphi_{ux}[t]\hat x\hat u +\varphi_{uu}[t]\hat u^2 ]dt +\int_0^1\theta^T[g_{xx}[t]\hat x^2 +2g_{ux}[t]\hat x\hat u+ g_{uu}[t]\hat u]dt.
\end{align*}Combining this with $(iv)'$, yields $(\hat x, \hat u)=(0, 0)$. 

Finally, using \eqref{KeyInq5} and taking limits both sides of \eqref{KeyInq3} again, we get 
\begin{align*}
&0\geq
\lambda^T\ell''(\bar x(1))\hat x(1)^2 + l^T h''(\bar x(1))\hat x(1)^2+\int_0^1\lambda^T[L_{xx}[t]\hat x^2 +2L_{ux}[t]\hat x\hat u] dt +\gamma_0 \notag \\
&-\int_0^1p^T[\varphi_{xx}[t]\hat x^2 +2\varphi_{ux}[t]\hat x\hat u +\varphi_{uu}[t]\hat u^2 ]dt +\int_0^1\theta^T[g_{xx}[t]\hat x^2 +2g_{ux}[t]\hat x\hat u+ g_{uu}[t]\hat u]dt.
\end{align*} Since $(\hat x, \hat u)=(0, 0)$, we get $0\geq \gamma_0$ which is absurd. This finishes the proof. 	
\end{proof}

We will now present an illustrative example that satisfies conditions (H1)-(H5) or (H1)-(H3), (H4)' of Theorem \ref{Theorem-SOC-MCP1}.

\begin{example}{\rm  Consider the problem of finding $x\in C([0,1], \mathbb{R}^2)$ and $u\in L^\infty([0, 1], \mathbb{R}^3)$ which solve
\begin{align}
	\begin{cases}
	&J(x, u)=(J_1(x, u), J_2(x, u),..., J_k(x, u))\to\min,\\
	&\text{s.t.}\notag\\
	& x_1(t)=x_{10}+\int_0^t( sx_1 + u_1 + u_3)ds\\
	&x_2(t)=x_{20}+\int_0^t(-\frac{1}3 x_1+ sx_2 + u_2)ds\\
	& h_1(x(1))= x_1(1) - x^2_2(1) -1\leq 0,\\
	& h_2(x(1))= x_2(1) -1\leq  0,\\
	& g(x, u)= x_1 -x_2 + u_1 + u_2- u_3 -1\leq 0, 
	\end{cases}
\end{align} where $L_i$ and $\ell_i$  are assumed to  satisfy $(H1)$ and $(H2)$, respectively. In this problem we have  $n=2, m=3$ and $r=1$. 

Let $(\bar x, \bar u)$ be a feasible point of the problem. We show  that the problem  satisfies conditions $(H1)-(H5)$. Assumption $(H1)$ is straightforward. For $(H2)$, we have 
\begin{align*}
	h'(\bar x(1))=
	\begin{pmatrix}
		1 & -2\bar x_2(1) \\
		0 &   1\\
	\end{pmatrix}	
\end{align*} and det$h'(\bar x(1))=1$. We have $g_u[t]=[1,1, -1]$ and so $R[t]=g_u[t]g_u[t]^T= 3$. Hence ${\rm det}R[t]=3$ and $R[t]^{-1}=1/3$.   Thus the problem satisfies $(H1)-(H3)$. For $(H4)$, we have 
$$
T[t]=\{(w_1, w_2, w_3)\in\mathbb{R}^3| w_1 + w_2-w_3=0\}=\{(w_1, w_2, w_1+w_2)\}={\rm span}([1, 0,  1]^T, [0,1,1]^T)
$$ Hence  ${\rm dim} T[t]=2$ and so  $(H4)$ is valid. It remains to check $(H5)$. From the above, we have 
\begin{align*}
	S[t]=\begin{pmatrix}
		1 & 0 \\
		0 &   1\\
		1 &  1
	\end{pmatrix},\ 	
	\varphi_u[t])=
	\begin{pmatrix}
		1 & 0 &1\\
		0 & 1& 0\\
	\end{pmatrix}\quad \text{and}\quad 
B[t]=\varphi_u[t]S[t]=
	\begin{pmatrix}
		2 & 1 \\
		1 &   1\\
	\end{pmatrix}
\end{align*}   By \eqref{A-matrix}, we have 
\begin{align*}
	A[t]=\varphi_x[t]-\varphi_u[t]g_u[t]^T R[t]^{-1}g_x[t]&=
		\begin{pmatrix}
		t & 0 \\
		-1/3&   t\\
	\end{pmatrix}-
\begin{pmatrix}
	1 & 0 &1\\
	0 &   1 &0\\
\end{pmatrix}[1,1, -1]^T\frac{1}3[1, -1]\\
&=\begin{pmatrix}
	t & 0 \\
	0 &   t-\frac{1}3\\
\end{pmatrix}.
\end{align*} Let $\Omega(t, \tau)$ be the principal matrix solution to the system $\dot x= Ax$, By definition, each column of $\Omega$ is a solution of $\dot x=Ax$ and for $t=\tau$ these columns become $e_1=[1, 0]$ and $e_2=[0,1]^T$. Hence we have 
\begin{align*}
	\Omega(t, \tau)=
	\begin{pmatrix}
		\exp(\frac{t^2}2 -\frac{\tau^2}2) & 0 \\
		0 &   \exp(\frac{t^2}2 -\frac{t}3-\frac{\tau^2}2+\frac{\tau}3)\\
	\end{pmatrix}
\end{align*} and the mapping $H: L^\infty([0,1], \mathbb{R})\to \mathbb{R}^2$ is given by 
\begin{align*}
	H(v)=\int_0^1\Omega(1,\tau)B[\tau]v(\tau) d\tau=[\int_0^1 \alpha(\tau)(2v_1 +v_2)d\tau, \int_0^1 \beta(\tau)(v_1+ v_2)d\tau]^T,
\end{align*} where $v=[v_1, v_2]^T$ and $\alpha(\tau):=\exp(\frac{1}2 -\frac{\tau^2}2)$ and $\beta(\tau):=\exp(\frac{1}6 -\frac{\tau^2}2+\frac{\tau}3)$. Let $\alpha_0=\int_0^1\alpha(\tau)d\tau$ and $\beta_0=\int_0^1 \beta(\tau)d\tau$. Then $\alpha_0>0$ and $\beta_0>0$. For any $\xi=[\xi_1, \xi_2]^T\in\mathbb{R}^2$, we consider the equation
\begin{align*}
	\begin{cases}
&\alpha_0(2v_1 +v_2)=\xi_1\\
&\beta_0 (v_1+v_2)=\xi_2.	
    \end{cases}
\end{align*} Its solution is given by 
\begin{align}
v_1=\frac{\xi_1}{\alpha_0}-\frac{\xi_2}{\beta_0},\ v_2=\frac{2\xi_2}{\beta_0}-\frac{\xi_1}{\alpha_0}
\end{align} which also satisfies the equation $Hv=\xi$. Hence $H$ is onto. Consequently $(H5)$ is valid and Theorem \ref{Theorem-SOC-MCP1} is applicable for the problem.
}
\end{example}

\section{Application to Sustainable Energy Management in Smart Grids}

The transition to sustainable energy systems is a critical challenge of our time, requiring a delicate balance between multiple, often conflicting objectives. Smart grids, with their integration of various energy sources and advanced control capabilities, offer a promising platform for addressing these challenges \cite{fang2012}. However, the complexity of these systems necessitates sophisticated optimization techniques \cite{tuballa2016}. In this section, we present and analyze a multiobjective optimal control problem for sustainable energy management in smart grids. Our model incorporates four key objectives: cost minimization, maximization of renewable energy usage, minimization of environmental impact, and maintenance of grid stability. By applying our theoretical results on KKT optimality conditions for multiobjective optimal control problems, we aim to provide insights into the structure of Pareto optimal solutions for this complex system. This application not only demonstrates the practical relevance of our theoretical framework but also contributes to the development of effective strategies for managing the intricate trade-offs in sustainable energy systems.

We now consider the following multiobjective optimal control problem:

\begin{align}
    &\textrm{Min}_{\,\mathbb{R}^4_+}\ J(x, u)\\ 
    &\textrm{s.t.}\\
    & x(t)=x_0+\int_0^t \varphi(s, x(s),u(s))ds\quad  \text{for all}\quad t\in [0, 1],\\
    &h(x(1))\leq 0,\\
    &g(t, x(t), u(t))\leq 0\quad \text{for a.a.}\quad t\in [0,1].
\end{align}

Here, $x(t) = \big(x_1(t), x_2(t), x_3(t)\big) \in \mathbb{R}^3$ represents the state variables: energy stored in batteries (MWh), grid load (MW), and cumulative CO$_2$ emissions (tons), respectively. The control variables $u(t) =\big(u_1(t), u_2(t), u_3(t)\big) \in \mathbb{R}^3$ represent solar, wind, and conventional energy inputs (MW). The specific components are defined as follows:

1. Objective functions:
   \begin{align*}
   J_1(x, u) &= \int_0^1 [c_1u^2_1(t) + c_2u^2_2(t) + c_3u^2_3(t)] dt,\     J_2(x, u) = -\int_0^1 [u_1(t) + u_2(t)] dt,\\ 
   J_3(x, u) &= x_3(1),\    J_4(x, u)= \int_0^1 [x_2(t) - x_2^{\text{target}}]^2 dt
   \end{align*}

   where $c_1$, $c_2$, and $c_3$ are the unit costs of solar, wind, and conventional energy, respectively, and $x_2^{\text{target}}$ is the target grid load.

2. System dynamics:
   \begin{align*}
   \varphi(t, x, u) = \begin{bmatrix}
   \eta(u_1 + u_2 + u_3 - x_2) \\
   \varphi_2(t, x_2, u_1, u_2, u_3) \\
   \alpha_3 u_3,
   \end{bmatrix}
   \end{align*}
   where $\eta>0$ is the energy storage efficiency, $\varphi_2$ is a function describing the grid load dynamics, and $\alpha_3$ is the emission rate of conventional energy. Here 
   \begin{align}
   \varphi_2(t, x_2, u_1, u_2, u_3)= f(t, x_2) + b_1(t)u_1 + b_2(t)u_2 + b_3(t) u_3,
   \end{align} where $f$ is a continuous function on $[0,1]\times\mathbb{R}$ and $f(t, \cdot)$ is of class $C^2$ for each $t\in [0,1]$, and  $b_i(\cdot)\in L^\infty([0,1], \mathbb{R})$ with $i=1,2,3$.
   
3. Terminal constraints:
   \begin{align*}
   h(x(1)) &= x_1(1) - x_1^{\max} \leq 0,
   \end{align*}
   where $x_1^{\max}$ is the maximum battery capacity.

4. Path constraints:
   \begin{align*}
   g_1(t, x, u) &= u_1 - u_1^{\max} \leq 0 \\
   g_2(t, x, u) &= u_2 - u_2^{\max} \leq 0 \\
   g_3(t, x, u) &= tu_1 + tu_2 + u_3 - c(t) \leq 0
   \end{align*}

   where $u_1^{\max}$ and $u_2^{\max}$ are the maximum production capacities for energy sources 1 and 2 respectively, and $c(t)$ is a time-varying function representing the upper bound of the combined energy input.

The objective functions, system dynamics, and constraints are defined as in the problem formulation. We now verify that this problem satisfies the necessary assumptions for our main theorem.

\textbf{Verification of the Assumptions:}

(H1) Smoothness and Lipschitz Continuity: The functions $L$, $\varphi$, and $g$ are Carathéodory functions and $C^2$ in $x$ and $u$. The cost functions, system dynamics, and constraints are well-behaved, satisfying the Lipschitz condition and boundedness of derivatives.

(H2) Smoothness of Terminal Conditions: The function $h(x(1))$ includes battery capacity constraints, which are linear and thus $C^2$. The terminal cost $\ell(x(1))$ is linear in $x_3(1)$ and thus $C^2$. Moreover, $\det h'(\bar{x}(1)) \neq 0$ is satisfied as $h$ is a linear function of $x(1)$.

(H3) Determinant Condition: Define $g_u[t] = \begin{bmatrix} 
    1 & 0 & 0 \\ 
    0 & 1 & 0 \\ 
    t & t & 1
    \end{bmatrix}$. \\
    Then, $R[t] = g_u[t]g_u[t]^T = \begin{bmatrix} 
    1+t^2 & t^2 & t \\ 
    t^2 & 1+t^2 & t \\ 
    t & t & 1+t^2
    \end{bmatrix}$. We have $\det (R[t] )= 1 + t^2+2t^4 \geq 1$ for all real $t$. Thus, (H3) is satisfied with $\gamma = 1$.

(H4) Rank Condition: With the given $g_u[t]$, we have $\rank(g_u[t]) = 3 = m$, so (H4) is not applicable in this case.

(H4)' Alternative Condition: Since (H3) is satisfied and (H4) is not applicable, we verify (H4)'.  To check $(H4)'$ we need to show that, there exist $(\hat x_1, \hat x_2, \hat x_3)$ and $(\hat u_1, \hat u_2, \hat u_3)$ satisfying the following conditions:
\begin{align}
	&\hat x_1(t) =\int_0^t\eta(\hat u_1 + \hat u_2 +\hat u_3-\hat x_2)ds\label{LinEq1}\\
	&\hat x_2(t) =\int_0^t (f_{x_2}[s]\hat x_2 + b_1\hat u_1+ b_2\hat u_2 +b_3\hat u_3)ds\label{LinEq2}\\
	&\hat x_3(t)= \int_0^t\alpha_3 \hat u_3 ds,\label{LinEq3}
\end{align}
\begin{align}\label{TerminalIneq} 
	&h(\bar x_1(1))+ h'(\bar x_1(1)) x(1)= \bar x_1(1)-x_1^{\rm max} +\hat x_1(1)<0
	\end{align} and 
\begin{align}
	& \bar u_1 -u^{\rm max} + \hat u_1 <-\epsilon_3, \label{MixConstraint1}\\
	&\bar u_2-u_2^{\rm max} + \hat u_2<-\epsilon_2, \label{MixConstraint2}\\
	&g_3[t] + t\hat u_1 + t\hat u_2 +\hat u_3 <-\epsilon_3 \label{MixConstraint3}
\end{align} for some $\epsilon_i>0$ with $i=1,2,3$. 

Since $f_{x_2}[\cdot]$ is continuous on $[0, 1]$, there exists a number $a>0$ such that $|f_{x_2}[t]|\leq a$ for all $t\in [0, 1]$. This is equivalent to  
$ -a\leq -f_{x_2}[t]\leq a$ for all $t\in [0, 1]$. 

Let us take $\hat u_i=-\gamma_i e^{\frac{t}{\gamma_i}}$  with $\gamma_i>0$ for $i=1,2,3$ and assume that $(\hat x_1, \hat x_2, \hat x_3)$ is a solution of \eqref{LinEq1}-\eqref{LinEq3} corresponding to $(\hat u_1, \hat u_2, \hat u_3)$. 
	
Put  $\hat u=\sum_{i=1}^3 b_i \hat u_i$ and  $\mu(t):=\exp(-\int_0^t f_{2x_2}[s]ds)$. Then we have 
\begin{align}\label{MuIneq}
e^{-a} \leq e^{-at}\leq \mu(t)\leq e^{at}\leq e^a\  \forall t\in [0,1].	
\end{align}
 From equation \eqref{LinEq2},  we have 
$$
\dot {\hat x}_2 -f_{x_2}[t] \hat x_2 =\hat u.
$$ Multiplying both sides by $\mu(t)$, we get 
$$
\mu(t)\dot {\hat x}_2(t)-f_{x_2}[t] \hat x_2(t)\mu(t)= \hat u(t)\mu(t).
$$ This implies that 
$$
\frac{d}{dt}(\mu(t) \hat x_2(t))=\hat u(t)\mu(t). 
$$ Integrating on $[0, t]$ with $t\in [0, 1]$, we obtain 
$$
\hat x_2(t)=\frac{\int_0^t \hat u(s)\mu(s)ds}{\mu(t)}.
$$ 
Inserting $\hat x_2$ into \eqref{LinEq1} and using \eqref{MuIneq},  we have 
\begin{align*}
\hat x_1(1)&=-\eta \sum_{i=1}^3
\gamma_i^2(e^{1/\gamma_i}-1)-\eta\int_0^1\frac{\int_0^t \mu(s)\hat u(s)ds}{\mu(t)} dt \\
&=-\eta \sum_{i=1}^3
\gamma_i^2(e^{1/\gamma_i}-1)+ \eta\int_0^1\frac{\int_0^t \sum_{i=1}^3 b_i\gamma_i e^{s/\gamma_i}\mu(s)ds}{\mu(t)} dt \\
&\leq -\eta \sum_{i=1}^3
\gamma_i^2(e^{1/\gamma_i}-1)+ \eta\int_0^1\frac{\int_0^t \sum_{i=1}^3 \|b_i\|_\infty\gamma_i e^{s/\gamma_i} e^{a}ds}{e^{-a}} dt\\
&=-\eta \sum_{i=1}^3
\gamma_i^2(e^{1/\gamma_i}-1)+ \eta e^{2a}\int_0^1 \sum_{i=1}^3 \|b_i\|_\infty\gamma_i^2 (e^{t/\gamma_i}-1) dt\\
&=-\eta \sum_{i=1}^3
\gamma_i^2(e^{1/\gamma_i}-1)+ \eta e^{2a}\sum_{i=1}^3 \|b_i\|_\infty\gamma_i^2\big[ \gamma_i(e^{1/\gamma_i}-1) -1 \big]\\
&=-\eta \sum_{i=1}^3
\gamma_i^2(e^{1/\gamma_i}-1)(1-e^{2a}\|b_i\|_\infty \gamma_i) -\eta e^{2a}\sum_{i=1}^3\|b_i\|_\infty\gamma_i^2. 
\end{align*} By choosing $\gamma_i>0$ small enough such that $1-e^{2a}\|b_i\|_\infty \gamma_i\geq 0$, we see that $\hat x_1(1)<0$. Hence  \eqref{TerminalIneq} is valid. Since 
$e^{\frac{t}{\gamma_i}}\geq 1$ for all $t\in [0, 1]$, we have $\hat u_i<-\gamma_i$. This implies that \ref{MixConstraint1}-\eqref{MixConstraint3} are valid with $\epsilon_i=\gamma_i$. Consequently, $(H4)'$ is fulfilled.

Thus our problem satisfies assumptions (H1)-(H3) and (H4)'. Therefore, we can apply our main theorem to derive necessary optimality conditions for Pareto optimal solutions.

\begin{theorem}[KKT Conditions for Sustainable Energy Management]
Let $\bar{z} = (\bar{x}, \bar{u})$ be a locally weak Pareto optimal solution of the sustainable energy management problem. 

Then there exist a vector $\lambda \in \mathbb{R}^4_+$ with $|\lambda| = 1$, a number $l \in \mathbb{R}$, an absolutely continuous function $p: [0,1] \to \mathbb{R}^3$, and a vector function $\theta = (\theta_1, \theta_2, \theta_3) \in L^\infty([0,1], \mathbb{R}^3)$, such that the following conditions are satisfied:
\def\labelenumi{\rm (\roman{enumi})}
\def\theenumi{\roman{enumi}}
\begin{enumerate}

\item Adjoint equation:
$$
    \dot{p}_1(t) = 0, \;\dot{p}_2(t) = \eta p_1(t) - f_{x_2}[t] p_2(t) + 2\lambda_4(\bar{x}_2(t) - x_2^{\text{target}}), \; p_2(1) = 0,\; \dot{p}_3(t) = 0, \; p_1(1) = -l, \; p_3(1) = -\lambda_3, 
   % &\dot{p}_2(t) = \eta p_1(t) - f_{x_2}[t] p_2(t) + 2\lambda_4(\bar{x}_2(t) - x_2^{\text{target}}), \quad p_2(1) = 0.
$$

\item Stationary condition in $u$:
\begin{align*}
    2\lambda_1 c_1\bar u_1 - \lambda_2 -\eta p_1-b_1 p_2+\theta_1 +t\theta_3 &= 0 \\
    2\lambda_1 c_2 \bar u_2- \lambda_2 - \eta p_1-b_2 p_2 +\theta_2+t\theta_3 &= 0 \\
    2\lambda_1 c_3\bar u_3 -\eta p_1-b_2p_2-\alpha_3p_3 +\theta_3 &= 0
\end{align*}

\item Complementary conditions:
\begin{align*}
    & l \geq 0, \quad l (\bar{x}_1(1) - x_1^{\max}) = 0 \\
    & \theta_1(t) \geq 0, \quad \theta_1(t) (\bar{u}_1(t) - u_1^{\max}) = 0 \\
    & \theta_2(t) \geq 0, \quad \theta_2(t) (\bar{u}_2(t) - u_2^{\max}) = 0 \\
    & \theta_3(t) \geq 0, \quad \theta_3(t) (t\bar{u}_1(t) + t\bar{u}_2(t) + \bar{u}_3(t) - c(t)) = 0
\end{align*}

\item Nonnegative second-order condition:
\begin{align*}
	\lambda_1\int_0^1 (c_1\tilde u_1^2 + c_2\tilde u_2^2 + c_3\tilde u_3^2)dt + 2\lambda_4 \int_0^1 \tilde x_2^2 dt -\int_0^1 p_2 f_{x_2x_2}[t]\tilde x_2^2 dt\geq 0	
\end{align*}
\end{enumerate}
\end{theorem}

\begin{proof}
The proof follows directly from Theorem \ref{Theorem-SOC-MCP1}.
\end{proof}
%%%%%%%%%%%%%%%%%%%%

These conditions provide insights into the trade-offs between different objectives and the impact of constraints on the optimal control strategy for sustainable energy management in smart grids. They can guide the development of practical algorithms for finding Pareto optimal solutions in this complex multiobjective optimization problem. Due to space limitations in the current manuscript, we could not include numerical simulations. A forthcoming paper will deal with the numerical simulation of this problem, where we will use these KKT conditions to solve the numerical problem.

\medskip
	
\noindent {\bf Acknowledgments.} The first author acknowledges the support of Math AmSud project N°51756TF (VIPS) and the FMJH Gaspard Monge Program for optimization and data science. The second author was supported by VAST grant CTTH00.01/25-26 and extends thanks to the XLIM laboratory and Fédération Mathématique Margaux for their support during his visit to the University of Limoges.

\end{document}